\documentclass[aos,preprint]{imsart}
\RequirePackage[OT1]{fontenc}
\RequirePackage{amsthm,amsmath,amsfonts}
\usepackage[T1]{fontenc}
\usepackage{color}
\usepackage{epsfig}
\usepackage[latin1]{inputenc}
\usepackage{rotating}
\usepackage{amssymb}
\usepackage{epsf}
\usepackage{graphicx}
\usepackage{bm}
\usepackage{psfrag}
\usepackage{subfigure}
\usepackage{url}
\usepackage{amsbsy}

\makeatletter
\let\NAT@parse\undefined
\makeatother

\usepackage[numbers]{natbib}
\RequirePackage[colorlinks,citecolor=blue,urlcolor=blue]{hyperref}
\startlocaldefs
\newcommand{\ie}{\emph{i.e.},~}
\newcommand{\eg}{\emph{e.g.},~}
\newcommand{\cf}{\emph{c.f.}~}

\newcommand{\sgn}{\mathrm{sgn}}
\newcommand{\supp}{\mathrm{supp}}
\newcommand{\dist}{\mathrm{dist}}
\newcommand{\Sgn}{\mathrm{Sgn}}

\newcommand{\ve}{\boldsymbol}

\renewcommand{\L}{{\mathcal L}}

\numberwithin{equation}{section}
\theoremstyle{plain}
\newtheorem{thm}{Theorem}[section]
\newtheorem{lem}{Lemma}[section]
\newtheorem{cor}{Corollary}[section]
\newtheorem{defi}{Definition}[section]
\newtheorem{remm}{Remark}[section]
\endlocaldefs

\begin{document}

\begin{frontmatter}
\title{On change point detection using the fused lasso method\protect\thanksref{T1}}
\runtitle{On change point detection}
\thankstext{T1}{This work was partially supported by the Swedish Research Council and the Linnaeus Center ACCESS at KTH. The research leading to these results has received funding from The European Research Council under the European Community's  Seventh Framework program (FP7 2007-2013) / ERC Grant Agreement N. 267381.}%

\begin{aug}
\author{\fnms{Cristian R.} \snm{Rojas}\ead[label=e1]{crro@kth.se}}
\and
\author{\fnms{Bo} \snm{Wahlberg}\ead[label=e2]{bo@kth.se}\ead[label=u1,url]{http://www.ee.kth.se}}
\runauthor{Rojas and Wahlberg}
\affiliation{KTH Royal Institute of Technology}
\address{Department of Automatic Control and ACCESS Linnaeus Center\\
School of Electrical Engineering\\
KTH Royal Institute of Technology\\
SE 100 44 Stockholm,Sweden\\
\printead{e1}
\phantom{E-mails:\ } %s:cristian.rojas@ee.kth.se; bo.wahlberg@ee.kth.se\ }
\printead*{e2} %\\
%\printead{u1}
}
\end{aug}

\begin{abstract}
In this paper we analyze the asymptotic properties of $\ell_1$ penalized maximum likelihood estimation of signals with piece-wise constant mean values and/or variances. The focus is on segmentation of a non-stationary time series with respect to changes in these model parameters.  This change point detection and estimation problem is also referred to as total variation denoising or  $\ell_1$ -mean filtering and has many important applications in most fields of science and engineering. We establish the (approximate) sparse consistency properties, including rate of convergence, of the so-called fused lasso signal approximator (FLSA). We show that this only holds if the sign of the corresponding consecutive changes are all different, and that this estimator is  otherwise incapable of correctly detecting the underlying sparsity pattern.
The key idea is to notice that the optimality conditions for this problem can be analyzed using techniques related to brownian bridge theory.
\end{abstract}

\begin{keyword}[class=AMS]
\kwd[Primary ]{62G08}
\kwd{62G20}
%\kwd[; secondary ]{60K35}
\end{keyword}
\begin{keyword}
\kwd{Fused Lasso, TV denoising, Fused Lasso Signal Approximator, change point detection.}
\end{keyword}
\end{frontmatter}

% ==============================
\section{Introduction}
Methods for estimating the mean, trend or variance of a stochastic process from time series data have applications in almost all areas of science  and engineering. For non-stationary data it is also important to detect abrupt changes in these parameters and to be able to segment the data into the corresponding stationary subsets. Such applications include failure detection and fault diagnosis. An important case is noise removal from a piecewise constant signal, for which there are a wide range of proposed denoising methods.  This problem is also known as step detection, and is a special case of change point detection. We refer to \cite{citeulike:9394683,citeulike:9394684} for a recent survey of such methods, including the method of one dimensional total variation (TV) denoising to be analyzed in the current paper. TV denoising was introduced in \cite{Rudin_1992}, and is closely related to the fused lasso method/signal approximator, \cite{citeulike:28022,jcgs.2010.09208}, the generalized lasso method, \cite{Tibshirani-Taylor-11}, and basis pursuit denoising with a heaviside dictionary, \cite{chen01}. The idea is use $\ell_1$ norm regularization to promote sparseness. Detection of changes in trends using this framework has recently been studied in the paper \cite{Kim-Koh-Boyd-Gorinevsky-09}, where $\ell_1$ trend filtering was introduced. This is an extension of the Hodrick-Prescott filter, \cite{citeulike:11033787}. We will analyze in detail the corresponding $\ell_1$ mean filtering algorithm.

The literature on $\ell_1$ regularized estimation methods in statistics is vast, and we have only given some relevant snap-shots of important references. We refer to the papers above for a more complete bibliography. There are several textbooks covering this topic, \eg \cite{Buhlmann-vandeGeer-11,citeulike:161814}.

In this paper we analyze the asymptotic properties of the fused lasso method/signal approximator (FLSA), focusing in its ability to approximately detect the location of the change points in the measured signal. Even though the support recovery properties of the fused lasso have already been studied in the literature (see the next subsection for references), and it has been established that as the number of samples increases, the fused lasso cannot recover the location of the change points exactly, our focus is on the \emph{approximate} recovery of these change points. In particular, we will show that this is possible under well defined circumstances, based on a interesting interpretation of the fused lasso estimate based on duality theory of convex optimization.

The paper is structured as follows: First we will give an intuitive introduction to the methods and the corresponding theory to be studied in the paper. The main results are presented in Section \ref{sec:2}. In order to improve the readability of the paper, most of the proofs are collected in the Appendix \ref{app:A}, except for those of the main results on consistency and inconsistency of the FLSA algorithm. In Section \ref{sec:3}, we discuss extensions including $\ell_1$ variance filtering. The paper is concluded in Section \ref{sec:4}.

% ----------------------------------------------------
\subsection{Problem formulation} \label{subsec:problem}
Consider the data $\{y_t,t=1,\ldots, N\}$ and assume that it has been generated by the non-stationary Gaussian stochastic process
\begin{align}
y_t \sim \mathcal{N}(m_t, 1),\;\mbox{where $m_{t+1}=m_t$ ``often''.}
\end{align}
The problem is now to estimate the means $m_t$, $t=1,\dots, N$, from the given data. To start with we have simplified the formulation by assuming a given fixed variance. This assumption will be relaxed in Section \ref{sec:3}. In order to solve this problem we first need to specify what we mean by ``often''.  This could be done by specifying the probability of a change and then using for example multiple model estimation methods \cite{gustafsson-2000a}. Here we will just assume that the mean value function is piecewise constant as a function of time $t$. One way to measure the variability of a sequence $\{m_t, t=1, \dots, N\}$ is to calculate its Total Variation (TV):
$$
\sum_{t=2}^N |m_t-m_{t-1}|.
$$
This is the $\ell_1$-norm of the first-difference sequence and can be seen as a convex approximation/relaxation of counting the number of changes. The fit to the data is measured by the least squares cost function
$$
\frac{1}{2}\sum_{t=1}^N (y_t-m_t)^2,
$$
which is related to the Maximum Likelihood (ML) cost function for the normal distributed case.
The so-called $\ell_1$  mean filter, the TV denoising estimate or the FLSA (fused lasso signal approximator) is given by minimizing a convex combination of these two cost functions,
\begin{align} \label{eq:tv}
\min_{m_1,\ldots, m_N} \frac{1}{2}\sum_{t=1}^N (y_t-m_t)^2+\lambda \sum_{t=2}^N |m_t-m_{t-1}|.
\end{align}
This is a convex optimization problem with only one design parameter, namely $\lambda$.
%We will call this the FLSA (fused lasso signal approximator).
The TV cost will promote solutions for which $m_t-m_{t-1}=0$, \emph{i.e.}, a piecewise constant estimate. The choice of the  regularization parameter $\lambda$ is very important and  provides a balance between the fit to the data and stressing the structure constraint.  The same idea can be used for the multivariate case, \ie for a vector valued stochastic process. The $\ell_1$ norm can then be replaced by a sum of norms and the vector estimate $\ve{m}_t \in \mathbb{R}^n$ is given by
$$
\min_{\ve{m}_1,\ldots \ve{m}_N} \frac{1}{2}\sum_{t=1}^N \|\ve{y}_t - \ve{m}_t\|_2^2 + \lambda \sum_{t=2}^N \|\ve{m}_t-\ve{m}_{t-1}\|_p,
$$
where typically $p=1, 2$. This is known as \emph{sum-of-norms  regularization} \cite{OhlssonLB:10}.

There are several known results and properties for the FLSA \eqref{eq:tv}, but also many open questions:
The convex optimization problem (\ref{eq:tv}) can be solved very efficiently with a wide range of methods. Standard interior point software can be used for moderate sized problems. For larger size problems (where $n \times N$ is large) first order methods, such as the alternating direction method of multipliers (ADMM), have nice scaling properties, see \cite{BoydPCPE11,Wahlberg12}. Our focus here will, however, be on theory rather than algorithms.

The key design parameter is $\lambda$. It is known that for sufficiently large values of $\lambda$, say $\lambda\geqslant \lambda_{\max}$, where $\lambda_{\max}$ will be defined later, the solution to (\ref{eq:tv}) is the empirical mean estimate
$$
\hat{m}_t = \frac{1}{N}\sum_{j=1}^Ny_j,
$$
\ie we only have one segment. We will re-derive the expression for $\lambda_{\max}$ in our analysis to follow. It is also known that the optimal solution $\hat{m}_t(\lambda)$ is piecewise linear as a function of $\lambda$, and that by reducing  $\lambda$ we only introduce new change points but will keep the change points obtained from larger values of $\lambda$. To be more precise, as $\lambda$ decreases, neither the transition times nor signs change, but only new transition times appear. This is referred to as the \emph{boundary lemma} in \cite{Tibshirani-Taylor-11}, it was first proven in \cite{Friedman-Hastie-Hoffling-Tibshirani-07} and further discussed in \cite{jcgs.2010.09208}.

It is known that problem \eqref{eq:tv} can be reformulated as a standard $\ell_1$ lasso problem
$$
\min_{\ve{x}} \|\ve{y}- \ve{A x} \|_2^2 + \lambda\| \ve{x} \|_1,
$$
for which necessary conditions for recovery of sparse solutions $\ve{x}$ are known. For example  the lasso estimator can only asymptotically recover the correct sparsity pattern if the $\ve{A}$ matrix satisfies the so-called \emph{irrepresentable condition}, \cite{citeulike:5177983}. However, even if these conditions do not hold the lasso estimator may still produce $\ell_2$ consistent estimates, see \cite{citeulike:3979760}.

Some asymptotic convergence properties of the fused lasso are given in  \cite{Rinaldo-09}, where conditions are derived under which the FLSA detects the exact location of the change points as $N \to \infty$; however, the results in \cite{Rinaldo-09} are not completely right, since, as discussed in \citep{Harchaoui-Levy_Leduc-08}, it is not possible to recover the exact change points even asymptotically, as the irrepresentable condition does not hold. On the other hand, $\ell_2$ consistency of the FLSA holds under more general conditions, \emph{c.f.} \citep{Harchaoui-Levy_Leduc-08}. In \citep{Qian-Jia-12}, a modified version of the FLSA is shown to recover the exact change points as the noise variance goes to zero.

% ----------------------------------------------------
\subsection{Optimality Conditions} \label{subsec:optim_cond}
Let us rewrite the FLSA problem (\ref{eq:tv}) as
\begin{align}
\begin{array}{cl}
\min\limits_{\{m_t\}_{t=1}^N,\{w_t\}_{t=2}^N} &\; \displaystyle \frac{1}{2}\sum_{t=1}^N (y_t-m_t)^2+\lambda \sum_{t=2}^N |w_t|\\
\text{s.t.} &\;  w_t=m_t-m_{t-1}, \quad t = 2, \dots, N,
\end{array}
\label{eq:primal}
\end{align}
and introduce the variables
\begin{equation}
\label{eq:zt}
z_{t}=  \sum_{j=1}^{t-1}(m_j-y_j),\; t=2,\dots , N.
\end{equation}
We will now show that $\{z_t, t = 2, \dots, N\}$ are the dual variables (prices) of problem \eqref{eq:primal} and that the Karush-Kuhn-Tucker (KKT) optimality conditions \citep{Luenberger-84} are
\begin{align} \label{eq:optcond}
&z_1=z_{N+1}=0,\nonumber\\
& |  z_{t}|\leqslant \lambda,\; t=2,\ldots , N,\nonumber\\
& |  z_{t}|<\lambda\;\mbox{(constant)}\quad \Rightarrow\: m_{t}=m_{t-1}, \\
&| z_{t_k}|=\lambda \;\mbox{(transition)} \quad \Rightarrow\;
 \sgn(m_{t_{k}}-m_{t_{k}-1}) = \sgn(z_{t_k}), \nonumber
\end{align}
where $t_0= 1 < t_1 < \cdots < t_M \leqslant N$ are the optimal transition times (change points).

To prove this, first differentiate the Lagrangian function (using vector notation for its argument)
\begin{align} \label{eq:lagr}
{\L}(\ve{m},\ve{w},\ve{z})=\frac{1}{2}\sum_{t=1}^N (y_t-m_t)^2+\lambda \sum_{t=2}^N |w_t|+ \sum_{t=2}^ N z_t( m_t - m_{t-1} - w_t)
\end{align}
with respect to $m_t$ to obtain
\begin{align} \label{eq:optm1}
-(y_1-m_1)-z_2 &= 0, \nonumber\\
-(y_t-m_t)+z_{t}-z_{t+1} &= 0, \quad t=2,\ldots, N-1, \\
-(y_N-m_N)+z_N &= 0. \nonumber
\end{align}
By adding up these equations, setting $z_1=z_{N+1}=0$, we obtain the expression \eqref{eq:zt} for $z_t$. The sub-gradient \citep{Rockafellar-70} of the Lagrangian \eqref{eq:lagr} with respect to $w_t$ equals
\begin{align*}
\lambda \Sgn(w_t)-z_{t},\quad t=2, \ldots, N,
\end{align*}
where
$$
\Sgn(w_t)  \in \left\{\begin{array}{ll}
\{-1\}, & w_t<0,\\
\, [-1,1],& w_t=0,\\
\{1\}, & w_t>0.
\end{array}\right.
$$
This gives the optimality conditions
$
z_t=\lambda \Sgn(w_t),\; t=2, \ldots, N,
$
which with $w_t=m_t-m_{t-1}$ (the constraint) proves the second part of \eqref{eq:optcond}.

An alternative way to derive these conditions is by means of the dual problem of \eqref{eq:primal}, namely
\begin{align} \label{eq:dual}
\begin{array}{cl}
\max\limits_{z_t} &\; \displaystyle - \frac{1}{2}\sum_{t=1}^N (z_{t+1}-z_t-y_t)^2\\
\text{s.t.} &| z_t|\leqslant \lambda ,\quad z_1=z_{N+1}=0.
\end{array}
\end{align}
The solution to the primal problem \eqref{eq:primal} can be recovered from \eqref{eq:optm1} as
$$
m_t=y_t+z_{t+1}-z_t.
$$
Also notice that the unconstrained solution to \eqref{eq:dual} is
$$
z_t = C_1+C_2t+\sum ( -y_j).
$$
These observations can be found in, \eg~\cite{Tibshirani-Taylor-11} and as pointed out in \cite{Condat12} they are related to the taut string algorithm in \cite{citeulike:6865334} already published in 2001. Our key observation is that the dual variables $z_t$ determined by \eqref{eq:optcond} can be viewed as a {\em Random Bridge}, the discrete equivalent of a Brownian Bridge, \ie a random walk with changing drift and end constraints,
\begin{align*}
z_1=0,\quad  z_t =  \sum_{j=1}^{t-1}[m_j-y_j],\quad z_{N+1}=0.
\end{align*}
The corresponding visual insight will help us to further analyze the properties of the FLSA.

\subsection{Lambda Max}\label{subsec:lambdamax}
To start with, assume that the optimal solution is the empirical mean with corresponding dual variables
$$
\hat{m}=\frac{1}{N}\sum_{j=1}^N y_j, \quad z_{t}(\hat{m})=  \sum_{j=1}^{t-1} (\hat{m}-y_j).
$$
Define
\begin{align} \label{eq:lambdamax}
\lambda_{\max}=\max_{k=1,\ldots ,N}|z_{k+1}(\hat{m})|=\max_{k=1,\ldots ,N}k\left|\frac{1}{N}\sum_{j=1}^N y_j-\frac{1}{k}\sum_{t=1}^k y_t\right|.
\end{align}
Assume that $\lambda>\lambda_{\max}$ in \eqref{eq:optcond}. Then
\begin{align*}
|z_{t}|\leqslant \max_{1\leqslant k\leqslant N}|z_{k+1}|=\lambda_{\rm max}<\lambda,
\end{align*}
which means that $|z_t|$ will never reach $\lambda$ and we can only have one segment. Thus the  optimal solution is the empirical mean $\hat{m}$. This simple analysis provides an intuitive explanation for the $\lambda_{\max}$ result, \eg derived in \cite{Kim-Koh-Boyd-Gorinevsky-09}.
\subsection{The Bias} \label{sec:bias}

We will now use the optimality conditions \eqref{eq:optcond} to obtain a more precise characterization of $z_t$. Let $\{ 1< t_1< \ldots <t_{M-1}\leqslant N\}$ be the transition (change point) times, {\em i.e.}~ $|z_{t_k}|=\lambda$. Then \eqref{eq:optcond} implies
\begin{align*}
z_t&=\sum_{j=1}^{t-1}(m_j-y_j),\quad  z_{t_k}=\lambda \sgn(m_{t_k}-m_{t_{k-1}}).
\end{align*}
Here we have used that $m_{t_{k}-1} =m_{t_{k-1}}$, since there is no transition in the interval $t_{k-1} < t < t_k$.
Subtracting these expressions gives
\begin{align*}
z_{t_{k+1}}- z_{t_{k}}&=\sum_{j=t_{k}}^{t_{k+1}-1}(m_{t_k}-y_j).
\end{align*}
We can now find the FLSA solution as
\begin{align*}
m_{1}&=\frac{1}{t_1 - 1}\sum_{j=1}^{t_1 - 1} y_j +\frac{\lambda}{t_{1}-1}\sgn(m_{t_1}-m_1),\\
m_{t_k}&=\frac{1}{t_{k+1}-t_{k}}\sum_{j=t_k}^{t_{k+1}-1} y_j +\frac{\lambda}{t_{k+1}-t_{k}}\left(\sgn[m_{t_{k+1}}-m_{t_{k}}]-
\sgn[m_{t_{k}}-m_{t_{k-1}}]\right).
\end{align*}
Since
$$
z_t=\lambda \sgn(m_{t_{k}}-m_{t_{k-1}}) + \sum_{t=t_k}^{t-1} (m_{t_k}-y_t),\quad t_k < t < t_{k+1},
$$
we can see that the bias part of $m_{t_k}$, that is,
\begin{equation}
\frac{\lambda}{t_{k+1}-t_{k}}\left(\sgn[m_{t_{k+1}}-m_{t_{k}}]-
\sgn[m_{t_{k}}-m_{t_{k-1}}]\right),
\label{eq:bias}
\end{equation}
will provide a drift term to $z_t$ in the interval $t_{k}<t<t_{k+1}$. This observation will be of utmost importance in the analysis to follow, which will be illustrated by the next two examples.

% ----------------------------------------------------
\subsection{Example 1}
Consider a signal $\{y_t\}$ which satisfies $y_t\sim \mathcal{N}(m_t,1)$, where $\{m_t\}$ is a piece-wise constant sequence:
\begin{align*}
m_t = \left\{
\begin{array}{rl}
1, & \text{if } 0< t \leqslant 1000,\\
2, & \text{if } 1000 < t \leqslant 2000,\\
1, & \text{if } 2000 < t \leqslant 4000.\\
\end{array} \right.
\end{align*}
Given $4000$ measurements $\{y_1, \ldots , y_{4000}\}$ plotted in Figure \ref{fig:data1}, we want to estimate the means $m_1, \ldots, m_{4000}$.
To solve problem~\eqref{eq:tv} with $\lambda=\lambda_{\max}/3$, \cf \eqref{eq:lambdamax}, a package for specifying and solving convex programs called \texttt{CVX}  \cite{Grant-Boyd-08} is used.
The data is plotted in Figure \ref{fig:data1}.
\begin{figure}[ht]
\begin{center}
\includegraphics[width =70mm]{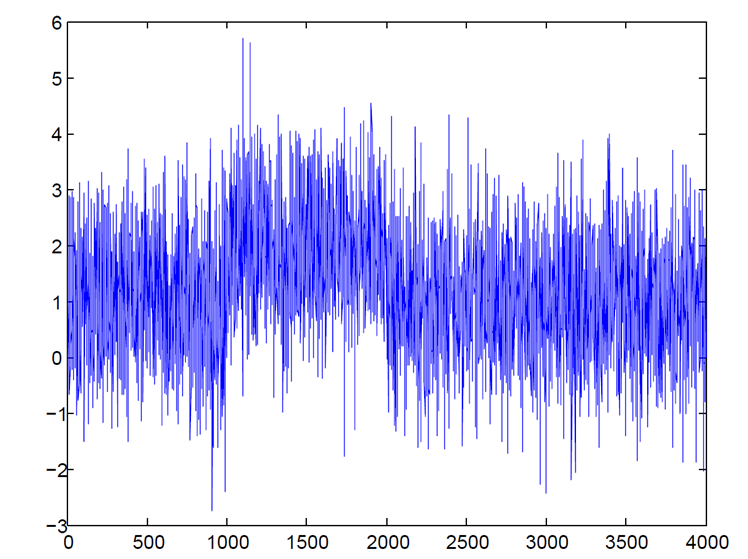}
\caption{Data Set 1.} %
\label{fig:data1} %
\vskip0.5\baselineskip
\includegraphics[width =70mm]{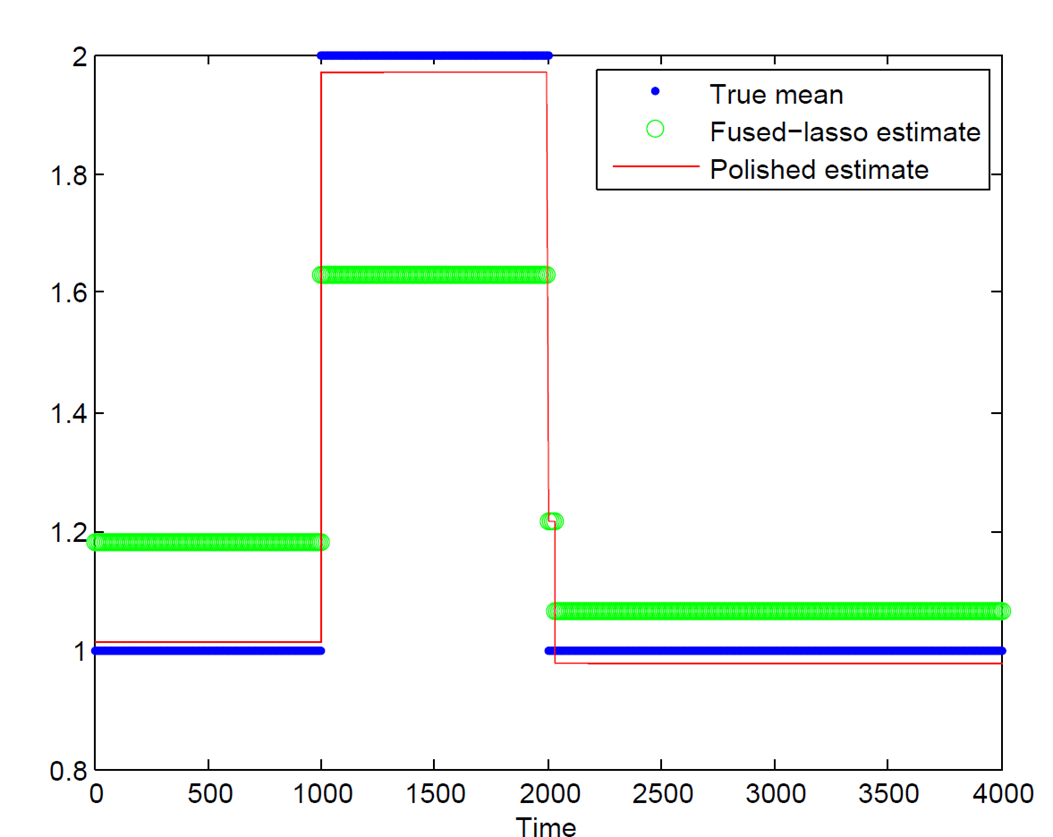}
\caption{Results for Data Set 1.} %
\label{fig:result1}
\end{center}
\end{figure}
Figure \ref{fig:result1} shows the true mean, the fused lasso (FLSA) estimate and the polished estimate, where the means have been re-estimated in the detected intervals. The results are good, in the sense that the change points have been correctly estimated (within reasonable precision).
\subsection{Example 2}
Let us now replace the true mean sequence by
\begin{align*}
m_t = \left\{
\begin{array}{rl}
1, & \text{if } 0< t \leqslant 1000,\\
2, & \text{if } 1000 < t \leqslant 2000,\\
3, & \text{if } 2000 < t \leqslant 4000.\\
\end{array} \right.
\end{align*}
The corresponding data set is plotted in Figure \ref{fig:data2}, and  the resulting estimates are shown in Figure \ref{fig:result2}. Here we have a detection error in
the second interval.
\vskip\baselineskip

\begin{figure}[h!]
\begin{center}
\includegraphics[width =70mm]{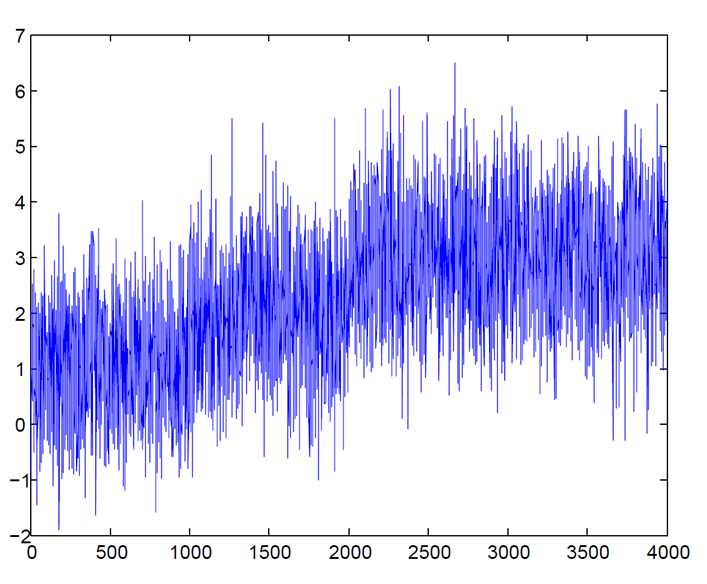}
\caption{Data Set 2.} %
\label{fig:data2} %
\vskip0.5\baselineskip
\includegraphics[width =70mm]{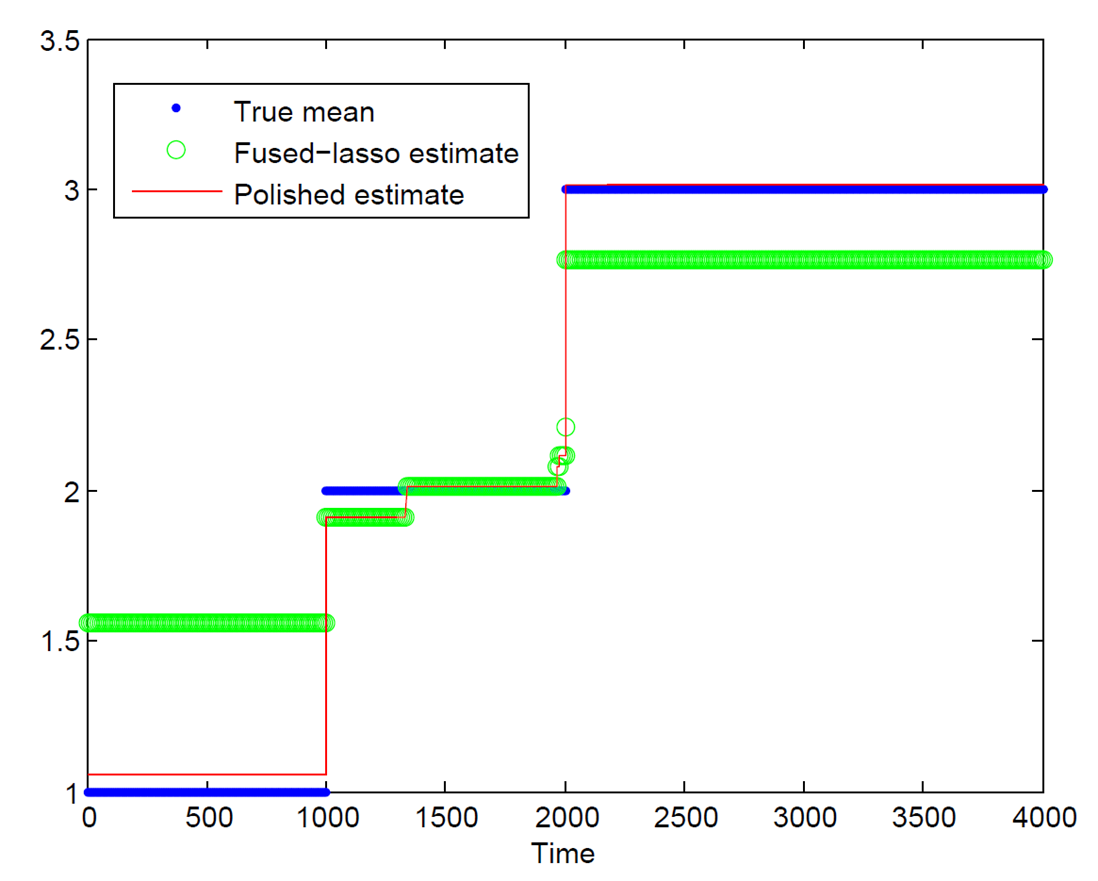}
\caption{Results for Data Set 2.} %
\label{fig:result2}
\end{center}
\end{figure}

\subsection{Explanation of Examples 1 and 2}
To explain the different outcomes let us also plot the corresponding optimal dual variables $\{z_t\}$ called ``random walk'' in Figures \ref{fig:dual1} and \ref{fig:dual2}. We notice that the incorrect detection in Example 2 can be explained by the bias term \eqref{eq:bias} that is zero for that example in the second interval. This means that the drift in the random walk is zero and hence the optimal solution is very sensitive to the noise; in particular, estimated change points appear every time  the random walk touches the $+\lambda$ boundary. We will call this the {\em stair-case problem}, since the reason is that the sign of the changes are both equal. This will be a key observation in the analysis of FLSA to follow in the next section.
\vskip0.1\baselineskip
\begin{figure}[h!]
\begin{center}
\includegraphics[width =80mm]{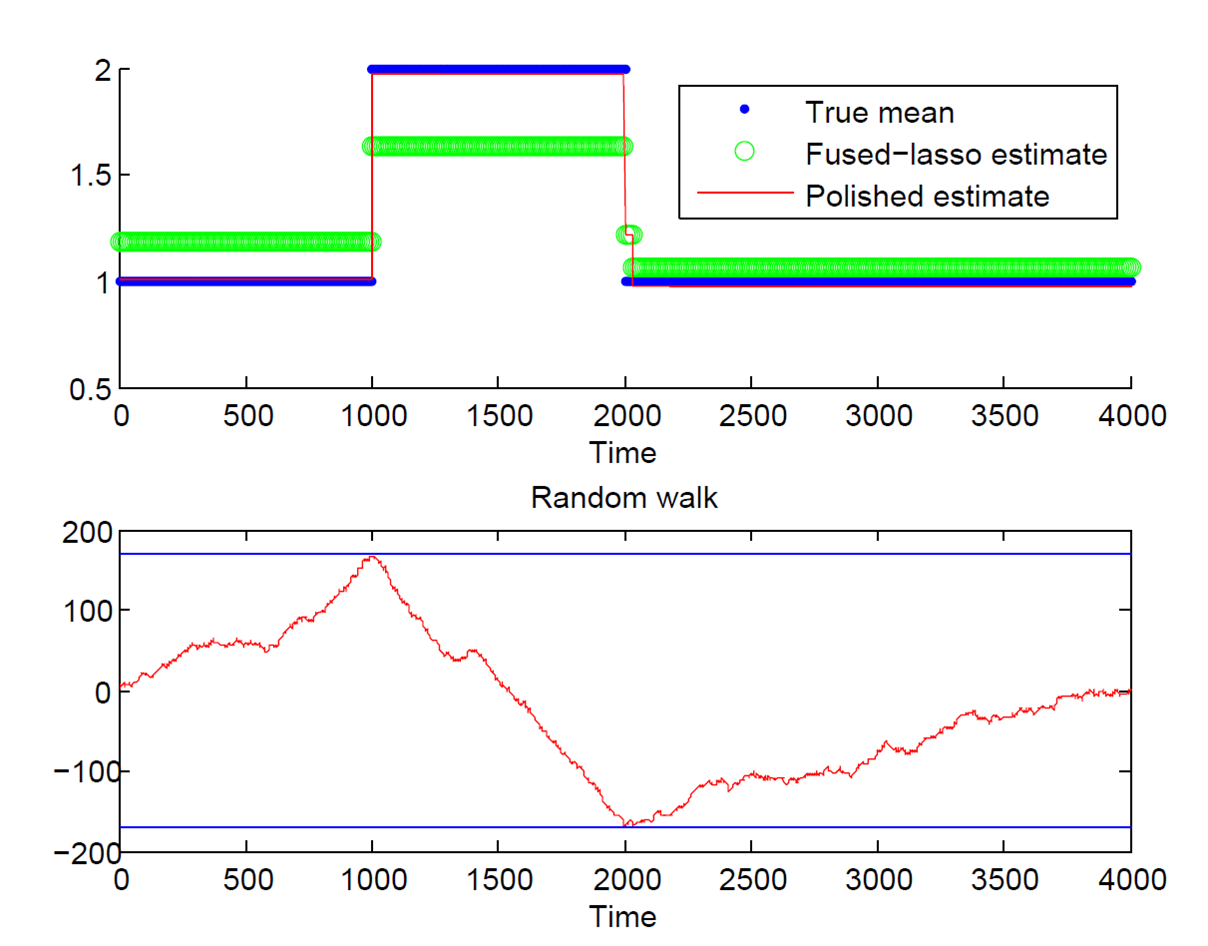}
\caption{Results  and  corresponding optimal dual variables for Example 1.} %
\label{fig:dual1}%
\vskip0.5\baselineskip
\includegraphics[width =80mm]{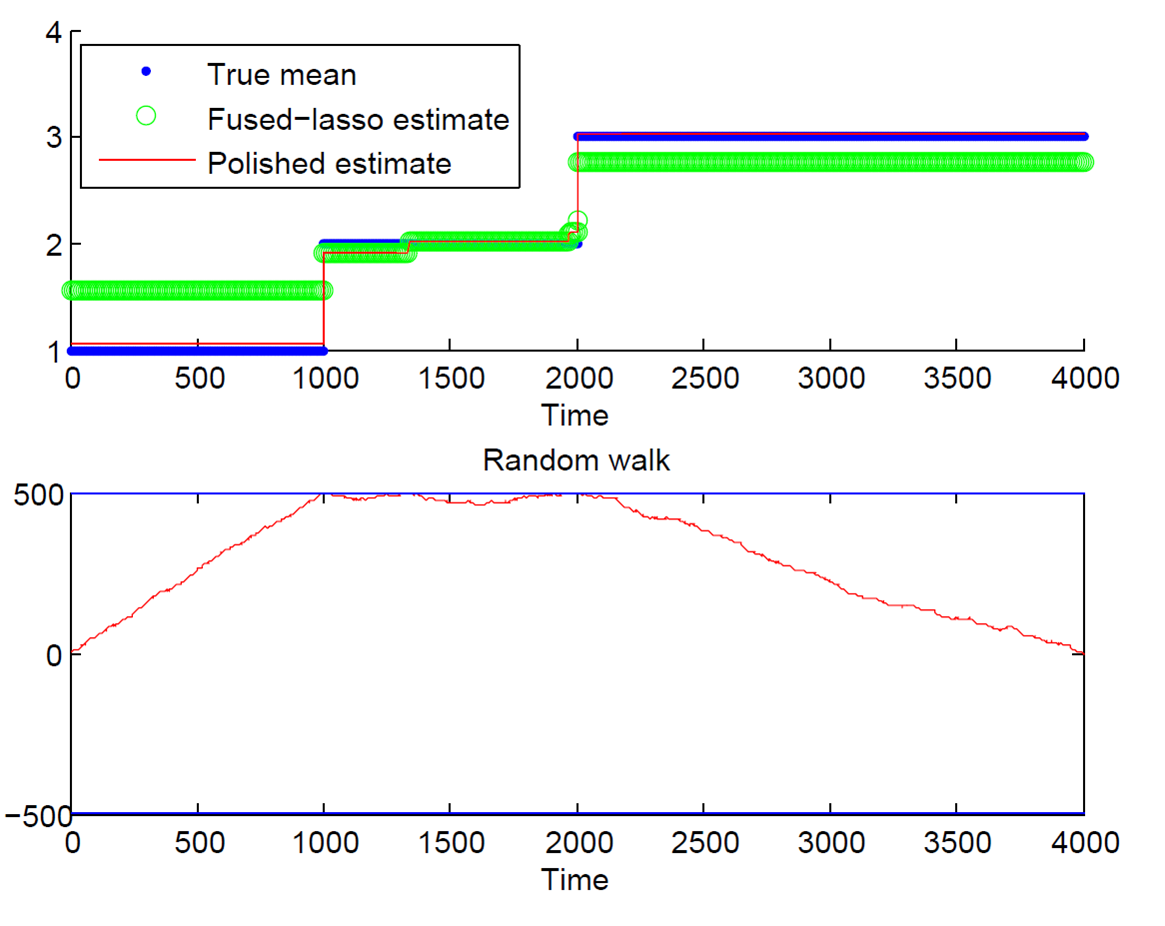}
\caption{Results  and  corresponding optimal dual variables for Example 2.} %
\label{fig:dual2}
\end{center}
\end{figure}

\newpage
% ==============================
\section{Consistency and Lack of Consistency}
\label{sec:2}
In this section we study the sparsity sign consistency (usually called ``sparsistency'') of the standard FLSA, given by
\begin{align} \label{eq:7}
\begin{array}{cl}
\min\limits_{m_1, \ldots, m_N} & \displaystyle \frac{1}{2} \displaystyle \sum\limits_{t = 1}^N (y_t - m_t)^2 + \lambda \sum\limits_{t = 2}^N |m_t - m_{t - 1}|.
\end{array}
\end{align}
where a proper stochastic description of $\{y_t\}$  will be postponed until later. We will rephrase some of the results of the previous section as lemmas.

As seen in Section~\ref{subsec:optim_cond}, the KKT conditions for the optimal solution of \eqref{eq:7} are given by
\begin{align} \label{eq:8}
\ve{m}- \ve{y} = \ve{A} \ve{z}
\end{align}
where $\ve{y} := [y_1 \; \cdots \; y_N]^T$,  $\ve{m} := [m_1 \; \cdots \; m_N]^T$, and
\begin{align} \label{eq:9}
&\ve{A} := \left[ \begin{array}{cccc}
1  &            &           & 0  \\
-1 &1          &           &     \\
    & \ddots & \ddots &    \\
    &            & -1       & 1  \\
0  &            &           & -1
\end{array} \right], \quad
\ve{z} := \left[ \begin{array}{c}
z_2 \\
\vdots \\
z_N
\end{array} \right], \\
&z_t \left\{ \begin{array}{ll}
=\lambda \sgn (m_t - m_{t - 1}), & \text{if } m_t \ne m_{t - 1} \\
\in [-\lambda, \lambda], & \text{otherwise.}
\end{array} \right. \nonumber
\end{align}
%(9)
Based on the KKT conditions, a simple characterization of the optimal solutions of problem \eqref{eq:7} can be derived.

\begin{lem}[Characterization of optimal solution] \label{lem:3}
$\ve{m} := [m_1 \; \cdots \; m_N]^T$ is an optimal solution of problem \eqref{eq:7} iff
\begin{align} \label{eq:10}
\max\limits_{1 \leqslant k \leqslant N - 1} \left| \sum\limits_{t = 1}^k (m_t - y_t) \right| &\leqslant \lambda \nonumber \\
\sum\limits_{t = 1}^{t_k - 1} (m_t - y_t) &= \lambda \sgn(m_{t_k + 1} - m_{t_k}),\quad k = 1, \ldots, M - 1 \\
\sum\limits_{t = 1}^N y_t &= \sum\limits_{t = 1}^N m_t. \nonumber
\end{align}
%(10)
where $1 < t_1 <  \cdots < t_{M - 1} \leqslant N$ are the values of $t \in \{1, \ldots, N\}$ at which $m_t \ne m_{t - 1}$, and $t_0 = 1$.
\end{lem}
By using
$$
z_1=0,\;z_j=\sum_{t=1}^{j-1}[m_t-y_t],\quad z_{N+1}=0,
$$
Lemma  \ref{lem:3} just gives the optimality conditions \eqref{eq:optcond} derived in Section \ref{subsec:optim_cond}. The bias result in Section \ref{sec:bias} is given in Lemma \ref{lem:4}. This lemma establishes the solution of problem \eqref{eq:7} when the location of the transition times $1 = t_0 < t_1 < \cdots < t_{M - 1} \leqslant N$ (see the notation in Lemma~\ref{lem:3}) and the transition signs are known.

\begin{lem}[Solution for known transition times] \label{lem:4}
Following the notation of Lemma~\ref{lem:3}, assume that the transition times $1 = t_0 < t_1 < \cdots < t_{M - 1} \leqslant N$ and the signs $s_k := \sgn(m_{t_k} - m_{t_{k - 1}})$ ($k = 1, \ldots, M - 1$) for an optimal solution $\ve{m}$ of problem \eqref{eq:7} are known. Then, $\ve{m}$ is given by
\begin{align*}
m_1 &= \frac{1}{t_1 - 1} \sum\limits_{t = 1}^{t_1 - 1} y_t + \frac{1}{t_1 - 1} \lambda s_1, \\ %
m_{t_k} &= \frac{1}{t_{k + 1} - t_k} \sum\limits_{t = t_k}^{t_{k + 1} - 1} y_t  + \frac{1}{t_{k + 1} - t_k} \lambda (s_{k + 1} - s_k); \quad k = 1, \ldots, M - 1,
\end{align*}
where $s_m := 0$.
\end{lem}

The following lemma establishes that, as $\lambda$ is decreased, neither the transition times $1 = t_0 < t_1 < \cdots < t_M = N + 1$ nor the signs $s_k = \sgn(m_{t_k} - m_{t_{k - 1}})$ change, but only new transition times appear. This lemma is essentially \citep{Friedman-Hastie-Hoffling-Tibshirani-07}[Proposition~2 (A2)] and is similar to the so-called ``boundary lemma'' of \citep{Tibshirani-Taylor-11}. For a proof, we refer the reader to those references.

\begin{lem}[Immobility of transition times] \label{lem:5}
Following the notation of Lemma~\ref{lem:4}, let $1 = t_0^\lambda  < t_1^\lambda < \cdots < t_{M_\lambda}^\lambda = N + 1$ be the transition times for a particular value of $\lambda$. Then, if $t$ is a transition time for $\lambda = \lambda_0$, \ie $t = t_k^{\lambda_0}$   for some $k \in \{1, \ldots, M_{\lambda_0} - 1\}$, then for every $\lambda < \lambda_0$,  $t = t_{k'}^\lambda$ for some $k' \in \{1, \ldots, M_\lambda - 1\}$.
\end{lem}

Lemmas~\ref{lem:3}, \ref{lem:4} and \ref{lem:5} give a nice interpretation of the FLSA estimate. Consider Figure~\ref{fig:interp}. Here, $y_t$ corresponds to the sketch of a noisy piece-wise constant signal, and $m_t$ is its FLSA estimate. Below this diagram, the dual variables $z_t = \sum_{j=1}^{t-1} [m_j - y_j]$ are displayed. According to Lemma~\ref{lem:3}, $z_t$ corresponds to a random bridge (\cf previous section), \emph{i.e.}, a conditioned random walk with drift whose end points are fixed at zero: $z_0 = z_N = 0$. Furthermore, $m_t$ is such that $|z_t|$ is bounded by $\lambda$, staying constant in segments where $|z_t| < \lambda$. $z_t$ takes the value $\lambda$ at those time instants where $m_t$ increases, and $-\lambda$ when $m_t$ decreases; in other words, $z_t$ is forced to take specific values at the end points of each segment where $m_t$ remains constant. To satisfy these end conditions, $m_t$ is subject to a bias, which is positive or negative depending on whether $m_t$ in the respective segment is a local minimum or maximum, respectively, \emph{c.f.} Lemma~\ref{lem:4}; in case the segment is part of an ascending or descending ``staircase'' (\emph{i.e.}, the segments lie between two change points where $m_t$ increases or decreases on both), the bias is zero. Notice in addition that the bias is higher for larger values of $\lambda$; in fact, as $\lambda$ is increased, the values of $m_t$ for consecutive segments get closer, until some critical value of $\lambda$ is reached, beyond which a change point disappears (\emph{i.e.}, some consecutive segments are fused together). However, according to Lemma~\ref{lem:5}, the location of the change points does not change with $\lambda$; they can merely disappear as $\lambda$ increases. As seen in Section~\ref{subsec:lambdamax}, for $\lambda \geqslant \lambda_{\max}$, all segments are fused together into one single segment.

\begin{figure}
\begin{center}
\includegraphics[height=0.7\columnwidth]{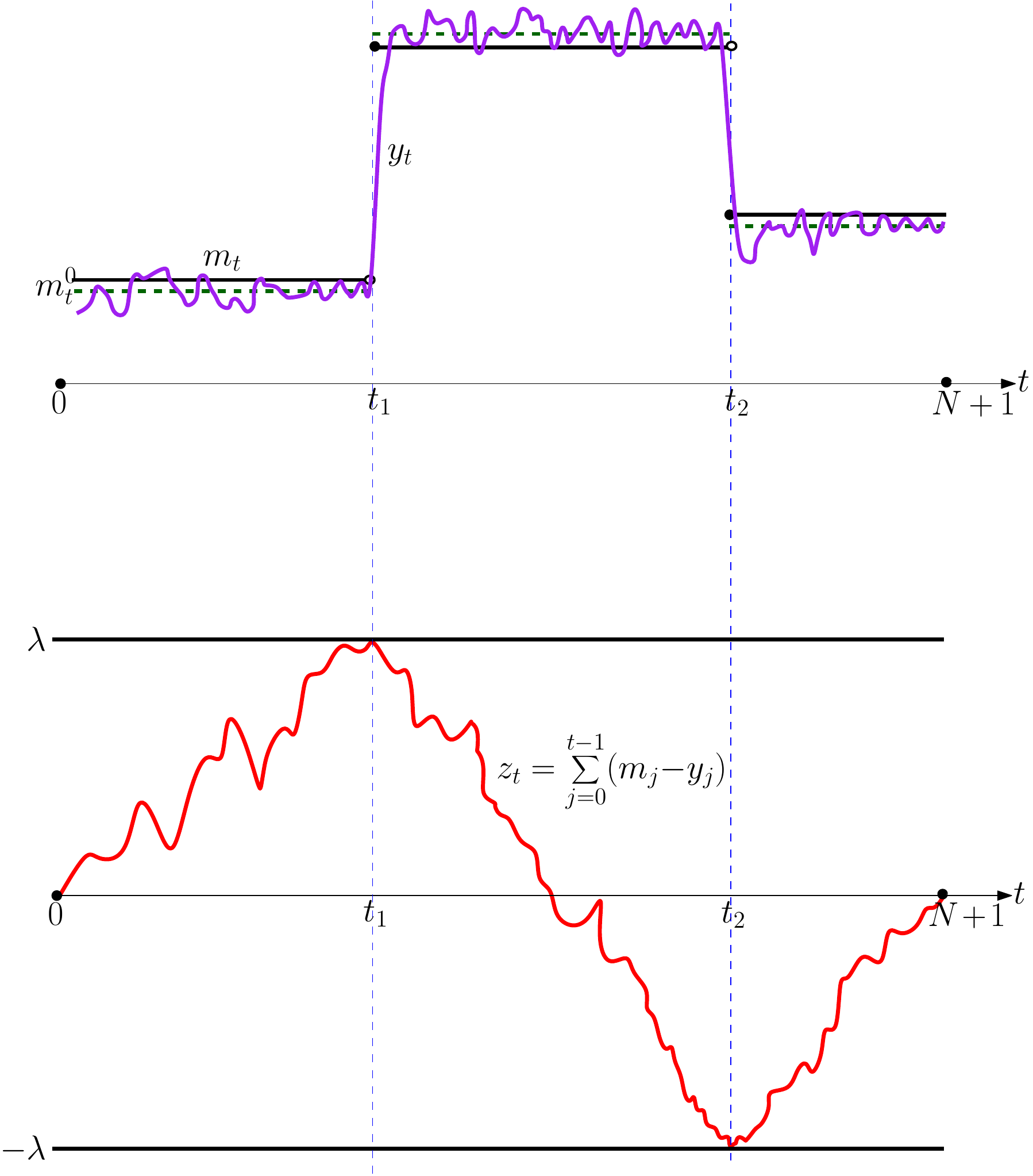}
\end{center}
\caption{Top: Piece-wise constant signal (dashed line) contaminated with noise (purple, solid line), and its FLSA estimate (solid, thick line). Bottom: Dual variable associated with the FLSA estimate.}
\label{fig:interp}
\end{figure}

The random bridge interpretation is very clear, and gives an intuitive explanation of conditions under which the FLSA provides a consistent estimate. We will digress for the moment, and study the consistency problem from the point of view of the so-called ``irrepresentable conditions''.

% ----------------------------------------------------
\subsection{The Irrepresentable Conditions}
One approach to study the consistency of FLSA is to reformulate it as a standard lasso problem, and then to work with the so-called irrepresentable condition. It is well known {\cite{Tibshirani-Taylor-11} that the FLSA can be formulated as an almost standard lasso problem of the form
\begin{align} \label{eq:11}
\min\limits_{\ve{\tilde{x}}} \;\frac{1}{2} \left\| \ve{y} - \ve{\tilde{A}} \ve{\tilde{x}} \right\|_2^2 + \lambda \sum\limits_{t = 2}^N \left| \tilde{x}_t \right|.
\end{align}
%(11)
To this end, we can define $\tilde{x}_1 := m_1$, $\tilde{x}_t := m_t - m_{t - 1}$ for $t = 2, \dots, N$, and
\begin{align*}
\ve{\tilde{A}} := \left[ \begin{array}{cccc}
1         & 0        & \cdots & 0         \\
1         & 1        & \ddots & \vdots \\
\vdots & \vdots & \ddots & 0        \\
1         &1         & \cdots &1
\end{array} \right] \in \mathbb{R}^{N \times N}.
\end{align*}
Equation~\eqref{eq:11} is not a standard lasso estimator yet, since the $\ell_1$ penalty involves only $N-1$ of the $N$ components in $\ve{\tilde{x}}$. However, it is possible to formulate the FLSA as a fully standard lasso, as shown in the following lemma.

\begin{lem}[Lasso equivalent form of the FLSA] \label{lem:6}
The FLSA can be reformulated as
\begin{align} \label{eq:12}
\min\limits_{\ve{x}} \; \frac{1}{2} \left\| \ve{\tilde{y}} - \ve{A x} \right\|_2^2 + \lambda \|\ve{x}\|_1,
\end{align}
%(12)
where $\ve{x} \in \mathbb{R}^{N - 1}$ is given by $x_t := \tilde{x}_{t + 1} = m_{t + 1} - m_t$ for $t = 1, \ldots, N - 1$, $\ve{\tilde{y}} := \ve{y} - (N^{-1} \sum\nolimits_{t = 1}^N y_t) \ve{1}_{N,1}$, and $\ve{A} \in \mathbb{R}^{N \times N - 1}$ is given by
\begin{align*}
A_{i,j} = \left\{ \begin{array}{ll}
\displaystyle \frac{j - N}{N}, & i \leqslant j \\
\displaystyle \frac{j}{N},       & i > j.
\end{array} \right.
\end{align*}
In addition, the solution of \eqref{eq:12} can be converted back into that of the FLSA by making
\begin{align*}
m_1 = \frac{1}{N} \sum\limits_{t = 1}^N y_t - \frac{1}{N} \sum\limits_{t = 1}^{N - 1} \sum\limits_{k = 1}^t x_k; \quad %
m_t = m_1 + \sum\limits_{k = 1}^{t - 1} x_k ,\quad t = 2, \ldots, N.
\end{align*}
Notice, finally, that the new data vector satisfies the equation $\ve{\tilde{y}} = \ve{A} \ve{x}_0 + \ve{\varepsilon'}$, where $\ve{x}_0 \in \mathbb{R}^{N - 1}$  is given by $(\ve{x}_0)_t := (\ve{m}_0)_{t + 1} - (\ve{m}_0)_t$ for $t = 1, \dots, N - 1$ and $\ve{\varepsilon '} := (\ve{I}_N - N^{-1} \ve{1}_{N,1} \ve{1}_{N,1}^T) \ve{\varepsilon}$, \ie the new noise vector has zero mean by construction, but it does not have independent components if $\ve{\varepsilon}$ had.
\end{lem}

%$$
%\min_{\ve{w}}\frac{1}{2}||\ve{\bar{y}} - \ve{A w}||_2^2+\lambda \|\ve{w}\|_1
%$$
%where
%%
%\begin{align*}
%\ve{w} &= [w_2 \ldots w_N]^T \in \mathbb{R}^{N-1} \\
%\bar{y}_t&=y_t-\frac{1}{N}\sum_{t=1}^N y_t,\quad \ve{\bar{y}} = [\bar{y}_1,\ldots,\bar{y}_N]^T\\
%\ve{A} & \in \mathbb{R}^{N \times N-1}, \quad
%A(i,j)=\left\{\begin{array}{ll} \displaystyle \frac{j-N}{N}, & i\leqslant j,\\ \displaystyle \frac{j}{N},& i>j. \end{array}\right.
%\end{align*}
%%

As an example, notice that for $N=4$ we have
\begin{align*}
\ve{A}=\frac{1}{4}\left[
\begin{matrix}
-3 & -2& -1\\ 1 & -2 & -1\\
1 & 2 & -1\\1 & 2 & 3
\end{matrix}\right].
\end{align*}

From the previous lemma, the asymptotic properties of the FLSA can be established in principle from the existing body of results on the lasso. The exact support recovery properties, in particular, are known to depend on the fulfillment of the so-called \emph{irrepresentable condition}~\cite{citeulike:5177983}. This condition, and its several variants, relies on a particular construction depending on the regressor matrix $\ve{A}$, which, for the case of the FLSA, is developed in the next two lemmas.

\begin{lem}[Normal matrix for the lasso equivalent of the FLSA] \label{lem:7}
For the lasso equivalent formulation of the FLSA, given by Lemma~\ref{lem:6}, the normal matrix $\ve{C} := \ve{\tilde{A}}^T \ve{\tilde{A}}$  is given by
\begin{align} \label{eq:13}
C_{ki} = C_{ik} = \frac{i (N - k)}{N}, \quad \text{for } i \leqslant k.
\end{align}
%(
\end{lem}

\begin{lem}[Interpolation property of the normal matrix] \label{lem:8}
Let $\ve{C} \in \mathbb{R}^{n \times n}$ be as in Lemma~\ref{lem:7}, and consider a set $K \subseteq \{1, \ldots, n\}$. Then, the matrix $\ve{X} := \ve{C}_{K^C,K} \ve{C}_{K,K}^{-1}$ (where $K^C := \{1, \ldots, n\} \backslash K$) has the following property\footnote{We use the following (abusive but convenient) notation: If $A$ is a given (ordered) set, then $A:\{1, \ldots, |A|\} \to A$ is a function that maps the index $i \in \{1, \ldots, |A|\}$ to $A(i)$, the respective element in $A$. Furthermore, we take $A(0) = 0$ and $A(i) = n + 1$ for $i > |A|$.}: for every $i \in \{1, \ldots, |K^C|\}$, $k \in \{1, \ldots, |K|\}$,
\begin{align*}
X_{i,k} = \left\{ \begin{array}{ll}
0, & K^C(i) \leqslant K(k - 1) \\
\displaystyle \frac{K^C(i) - K(k - 1)}{K(k) - K(k - 1)}, & K(k - 1) \leqslant K^C(i) \leqslant K(k) \\
\displaystyle \frac{K(k + 1) - K^C(i)}{K(k + 1) - K(k)}, & K(k) < K^C(i) \leqslant K(k + 1) \\
0, & K^C(i) > K(k + 1).
\end{array} \right.
\end{align*}
\end{lem}

The strong irrepresentable condition~\cite{citeulike:5177983} states that in order for the lasso to achieve support recovery, it is sufficient that $| \ve{C}_{K^C,K} \ve{C}_{K,K}^{-1} \ve{s}| < \delta$ for some $\delta < 1$ independent of the number of samples $N$, where $K$ is the support of the true $\ve{x}$, and $\ve{s} = \sgn(\ve{x}_K)$. Figure~\ref{fig:irrep} provides an interpretation of this condition for the FLSA, based on Lemma~\ref{lem:8}. Here, $m_t$ corresponds to a piece-wise constant signal, and $\ve{C}_{t,K} \ve{C}_{K,K}^{-1} \ve{s}$ is also plotted as a function of $t$; notice that this latter plot is consistent with Lemma~\ref{lem:8}, since $\ve{C}_{t,K} \ve{C}_{K,K}^{-1} \ve{s}$ is basically a linear combination (weighted by the entries of $\ve{s}$) of linear spline functions with knots at the change points of $m_t^o$. 

According to the strong irrepresentable condition, $\ve{C}_{t,K} \ve{C}_{K,K}^{-1} \ve{s}$ should be uniformly bounded in magnitude by some $\delta$ for every $t$ which is not a change point of $m_t^o$; as Figure~\ref{fig:irrep} shows that this is not possible, since $\ve{C}_{t,K} \ve{C}_{K,K}^{-1} \ve{s}$ approaches $\pm 1$ linearly at every change point. Notice the resemblance between the shapes of $\ve{C}_{t,K} \ve{C}_{K,K}^{-1} \ve{s}$ and the dual variables $z_t$, \cf Figure~\ref{fig:interp}. From this analogy, it is easy to see that the situation is even worse in the presence of a stair-case, \cf Figure~\ref{fig:dual2}, since in this case $|\ve{C}_{t,K} \ve{C}_{K,K}^{-1} \ve{s}| = 1$ for every $t$ in the segment between two change points of the same sign.

\begin{figure}
\begin{center}
\includegraphics[height=0.7\columnwidth]{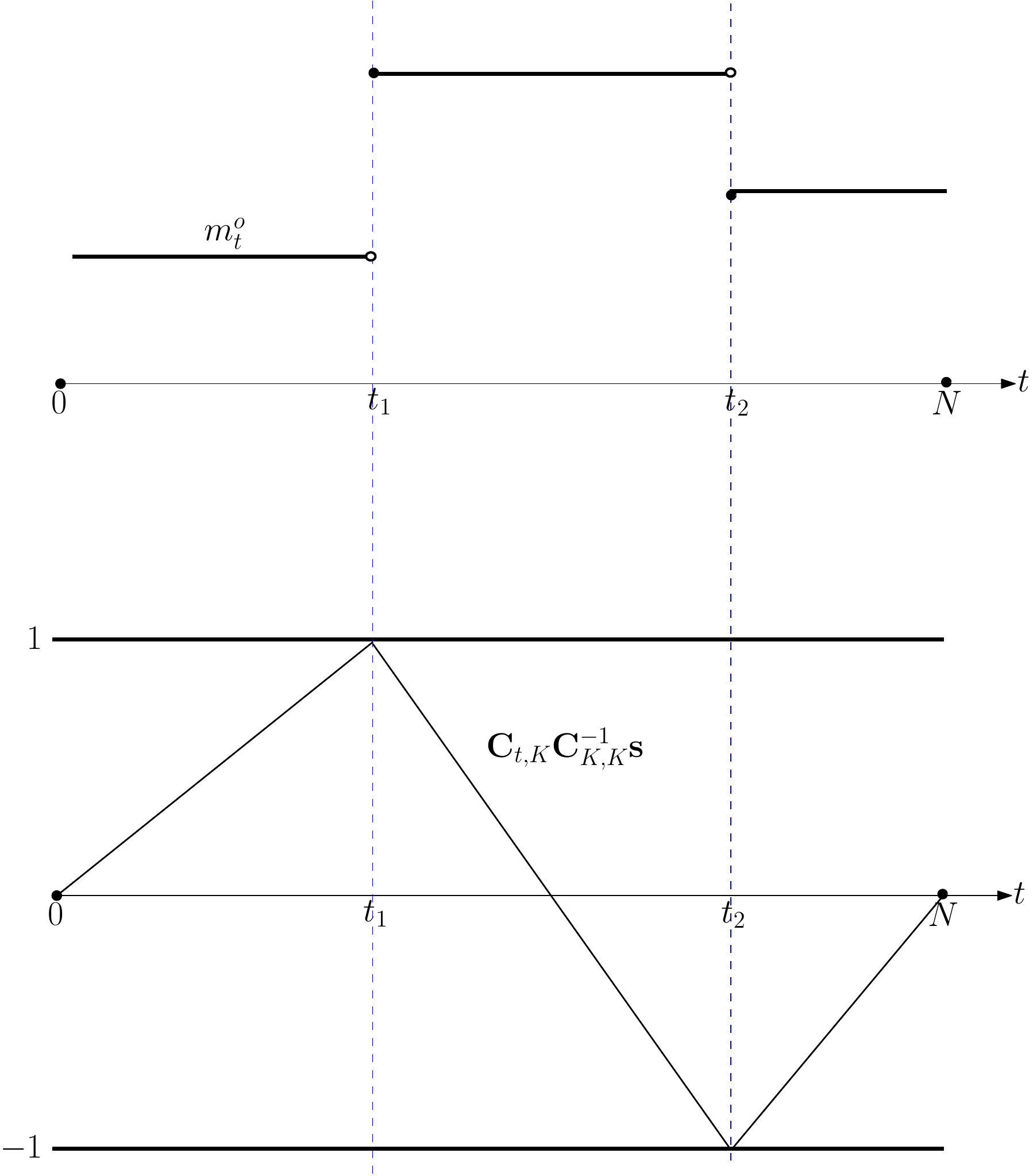}
\end{center}
\caption{Top: Piece-wise constant signal. Bottom: Plot of $\ve{C}_{t,K} \ve{C}_{K,K}^{-1} \ve{s}$, where $K$ is the set of values of $t$ for which $m_t \neq m_{t+1}$, and $\ve{s} = \sgn(m_{t+1} - m_t)$.}
\label{fig:irrep}
\end{figure}

While the strong irrepresentable condition is a sufficient criterion for support recovery of the lasso, other variants of this condition are indeed necessary for such property to hold. The reader is referred to \cite{Buhlmann-vandeGeer-11} for several interesting variants of the irrepresentable condition and related criteria. The following lemma presents a particular variant which is relevant to our analysis of the FLSA.

%{\bf Bo:} The notation in LEMMA~\ref{lem:8} is difficult to understand. Can you give a four by four example!
%\vskip\baselineskip
% {\bf Bo:} We need to show that
%\begin{align}
%| \ve{C}_{K^C,K} \ve{C}_{K,K}^{-1} \ve{s}| < \ve{1}\quad \mbox{IF}\quad s_i=(-1)^i
%\end{align}
%and otherwise
%\begin{align}
%| \ve{C}_{K^C,K} \ve{C}_{K,K}^{-1} \ve{s}| =\ve{1}
%\end{align}
%Is this trivial from the proof? I have not done the calculations due to the issue with the notation. This would also make the \ref{lem:10} more obvious
%\vskip\baselineskip

%\begin{lem}[Projection property in the Fused Lasso] \label{lem:9}
%Let $\ve{M} \in \mathbb{R}^{(N - 1) \times N}$ be given by
%%
%\begin{align*}
%\ve{M} = \ve{\tilde{A}}^T \{\ve{I} - \ve{\tilde{A}}_{:,K} [\ve{\tilde{A}}^T \ve{\tilde{A}}]_{K,K}^{-1} (\ve{\tilde{A}}^T)_{K,:}\},
%\end{align*}
%%
%where we have used the notation of Lemma~\ref{lem:6}, and $K \subseteq \{1, \ldots, n\}$ is an arbitrary set. Then, $\ve{M}_{k,:} = \ve{0}$ for every $k \in K$, and if $k \in (K(i), K(i + 1))$ and $l \in \{ 1, \ldots, N\} $ (where $i \in \{0, \ldots, |K| + 1\}$), then
%%
%\begin{align*}
%M_{k,l} = \left\{ \begin{array}{ll}
%\displaystyle \frac{k - K(i + 1)}{K(i + 1) - K(i)}, & K(i) < l \leqslant k \\
%\displaystyle \frac{k - K(i)}{K(i + 1) - K(i)},       & k < l \leqslant K(i + 1) \\
%0,                                                                     & l \notin (K(i), K(i + 1)].
%\end{array} \right.
%\end{align*}
%\end{lem}

\begin{lem}[Necessity of the irrepresentable condition of the lasso] \label{lem:10}
Consider the lasso problem
\begin{align} \label{eq:14}
\min\limits_{\ve{x} \in \mathbb{R}^n} \; \frac{1}{2} \|\ve{y} - \ve{A x}\|_2^2 + \lambda \|\ve{x}\|_1,
\end{align}
%(14)
where $\ve{y} \in \mathbb{R}^N$. Let $\ve{C} := \ve{A}^T \ve{A}$, consider a subset $K \subset \{1, \ldots, n\}$ and assume that the irrepresentable condition\footnote{The matrix $\ve{C}_{K,L}$ , where $K,L \subseteq \{1, \ldots, N\}$, is formed by taking the rows and columns of $\ve{C}$ indexed by $K$ and $L$ respectively.}
\begin{align} \label{eq:15}
| \ve{C}_{K^C,K} \ve{C}_{K,K}^{-1} \ve{s}| < \ve{1}
\end{align}
%(15)
does not hold component-wisely for some vector $\ve{s} \in \{-1, 1\} ^{|K|}$. If $\ve{x}_0 \in \mathbb{R}^n$ is some vector satisfying $(\ve{x}_0)_i = 0$ for every $i \notin K$, and $\sgn(\ve{x}_0)_{K(i)} = s_i$, and $\ve{y} = \ve{A} \ve{x}_0 + \ve{\varepsilon}$, where $\ve{\varepsilon} \in \mathbb{R}^N$ is such that $\delta := \min _{\ve{\alpha} \in \mathbb{R}^N \backslash \{\ve{0}\}} \min \{P[\ve{\alpha}^T \ve{\varepsilon} > 0], P[\ve{\alpha}^T \ve{\varepsilon} < 0]\} > 0$, then with probability at least $\delta$ the solution $\ve{x}$ of \eqref{eq:14} does not satisfy simultaneously

\begin{itemize}
\item ${x_i} = 0$ for every $i \notin K$,

\item $\sgn(x_{K(i)}) = s_i$,
\end{itemize}

for any value of $\lambda > 0$. In other words, the lasso cannot estimate the zero entries of $\ve{x}_0$ and the sign of its nonzero entries for all possible realizations of $\ve{\varepsilon}$, even if this noise is ``almost negligible''.
\end{lem}

This lemma states that, for many interesting noise distributions (including Gaussian and $\chi^2$ distributions, \cf Section~\ref{sec:3}) the FLSA cannot achieve exact support recovery. In the next sections, we will focus instead on the possibility or impossibility of achieving \emph{approximate} support set recovery with the FLSA. To this end, we require a new consistency concept, suitable for the study of approximate support set recovery:

\begin{defi} \label{def:eps-sign-consistency}
The FLSA estimate is said to be \emph{$\varepsilon$-sign-consistent} if, given any $\varepsilon, \gamma > 0$, no matter how small, there is an $N_0 \in \mathbb{N}$ such that, for all $N \geqslant N_0$, the probability that there is an $\lambda  > 0$ such that the optimal solution $\ve{m}_N$ of problem \eqref{eq:7} has all its change points inside an $\varepsilon N$-neighborhood of the change points of $\ve{m}_N^0$, and that for every change point of $\ve{m}_N^0$ there is a change point of $\ve{m}_N$ at a distance of at most $\varepsilon N$, is at least $1 - \gamma$. Otherwise, the FLSA estimate is said to be \emph{$\varepsilon$-sign-inconsistent}.
\end{defi}

Let $C(\ve{m}_N)$ and $C(\ve{m}_N^0)$ be the sets of change points of $\ve{m}_N$ and $\ve{m}_N^0$, respectively. Then, the statement in Definition~\ref{def:eps-sign-consistency} that $\ve{m}_N$ has all its change points inside an $\varepsilon N $-neighborhood of those of $\ve{m}_N^0$, and that for every change point of $\ve{m}_N^0$ there is a change point of $\ve{m}_N$ at a distance of at most $\varepsilon N$, can be formulated as: $\dist(C(\ve{m}_N),C(\ve{m}_N^0)) < \varepsilon N$, where
\begin{align*}
\dist(A,B) := \max \left\{ \max_{x \in A} \min_{y \in B} |x-y|, \max_{y \in B} \min_{x \in A} |x-y| \right\}
\end{align*}
for any sets $A, B \subset \{1, \dots, N\}$.

Notice that the standard notion of support set recovery (applied to the FLSA) is recovered by setting $\varepsilon = 0$. The definition of $\varepsilon$-sign-consistency is a more reasonable requirement for change point detection than exact change point recovery, since the latter is in general impossible to achieve. On the other hand, the analysis of this new property is more difficult than for support set recovery, due to the slack provided by $\varepsilon$. Notice in particular that $\varepsilon$-sign-consistency assumes a natural order (and topological) relation between the regressors of the $\ve{A}$ matrix associated with the lasso equivalent of FLSA, while most asymptotic results for the lasso and its variants do not consider any relation between its regressors.

% ----------------------------------------------------
\subsection{Inconsistency of the FLSA} \label{subsec:inconsistency}

%{\color{red}
%{\bf Bo:} In this subsection we start to use $x_t$ instead of $m_t$. Is this a problem? Seems a awful lot of places to change and easy to do mistakes.
%\vskip 1cm}

In this section we will analyze conditions under which the FLSA is inconsistent. To this end, we need the following basic lemma from fluctuation theory (for a proof, see \citep{Sparre-Andersen-53}[Theorem~3] or \citep{Spitzer-56}[pp. 328]):

\begin{lem}[Combinatorial lemma of fluctuation theory] \label{lem:inconsistency1}
Let $\ve{x} = [x_1\; \cdots $ $x_n]^T$ be a vector of real exchangeable continuous random variables, and
\begin{align*}
u_k := \sum\limits_{i = 1}^k x_i - \frac{k}{n} \sum\limits_{i = 1}^n x_i ,\quad k = 1, \ldots, n.
\end{align*}
Then, for $r = 0, 1, \dots, n-1$, the probability that exactly $r$ of the sums $u_k$ ($k=1, \dots, n$) is positive is $1/n$.
\end{lem}

The following result is also needed to establish inconsistency of the FLSA:

\begin{lem}[Lower bound on crossing probability] \label{lem:inconsistency2}
Let $x_t \sim \mathcal{N}(\mu, \sigma^2)$, $t = 1, \dots, N$, be independent random variables, and define $s_t := \sum_{i=1}^t x_i$, $t = 1, \dots, N$, where $s_0 := 0$. Then, given $\delta > 0$ small, there exist $\varepsilon_0 \in (0,1)$, $N_0 \in \mathbb{N}$ such that for all $N\geqslant N_0$ and $0 \leqslant \varepsilon \leqslant \varepsilon_0$ the probability
\begin{align*}
Q := P\left\{ s_t \geqslant \frac{t}{N} s_N \; \forall t \in \{1, \ldots, \lfloor \varepsilon N \rfloor\}  \cup \{N - \lfloor \varepsilon N \rfloor  + 1, \ldots ,N\} \right\}
\end{align*}
is lower bounded by $1/[(1 + \delta) \pi^2 \varepsilon N]$.
\end{lem}

Based on the previous lemmas, we can establish a negative result on the inconsistency of the FLSA when two or more successive change points occur in the same direction (\ie either upwards or downwards). The main ingredient of the proof is that the probability that a random walk, whose ends are fixed at $0$, takes only positive values is asymptotically very small; while this result is classical, the definition of $\varepsilon$-sign-consistency allows some slack in the location of the endpoints of the random walk, which adds some complications to the proof.

\begin{thm}[$\varepsilon$-sign-inconsistency] \label{thm:inconsistency}
Consider the notation of Lemma~\ref{lem:3}, and assume that the data satisfies an equation of the form $\ve{y}_N = \ve{m}_N^0 + \ve{\varepsilon}_N$, where $\ve{\varepsilon}_N$ is a random vector of independent and identically distributed normal components. Furthermore, assume that the following conditions hold:
\begin{itemize}
\item[(a)] $\min_{k = 2, \dots, M^0 + 1} (t_k^0 - t_{k - 1}^0) \geqslant c N$.
\item[(b)] There is a pair of consecutive signs, $s_k^0$ and $s_{k+1}^0$, which are equal.
\end{itemize}
Then, the FLSA is $\varepsilon$-sign-inconsistent (\cf Definition~\ref{def:eps-sign-consistency}). %, \ie given $\varepsilon, \gamma > 0$ sufficiently small, there is an $N_0 \in \mathbb{N}$ such that, for all $N \geqslant N_0$, the probability that there is no $\lambda  > 0$ such that the optimal solution $\ve{m}_N$ of problem \eqref{eq:7} has all its change points inside an $\varepsilon$-neighborhood of the change points of $\ve{m}_N^0$ is at least $\gamma$.
\end{thm}

\begin{proof}
Fix $\varepsilon > 0$ sufficiently small, with $\varepsilon < c$. Let us assume, without loss of generality, that $s_k^0 = s_{k + 1}^0 = -1$, and that $\lambda > 0$ is chosen so that $\ve{m}_N$ has change points within the $\varepsilon$-neighborhood of each $t_i^0$. Let $Q(t_i, t_j)$  be the event that $t_i$ is the largest integer in $(t_k^0 - \varepsilon N/2, t_k^0 + \varepsilon N/2)$ and $t_j$ is the smallest integer in $(t_{k + 1}^0 - \varepsilon N/2, t_{k + 1}^0 + \varepsilon N/2)$. Conditioned on $Q(t_i, t_j)$, and assuming that there is no change point between $t_i$ and $t_j$, notice that, by Lemma~\ref{lem:4},
\begin{align*}
m_t = \frac{1}{t_j - t_i} \sum\limits_{s = t_k}^{t_{k + 1} - 1} y_s, \quad t = t_k, \dots, t_{k + 1} - 1,
\end{align*}
hence the condition for $\ve{m}_N$ not having any change points between $t_i$ and $t_j$ is that
\begin{align*}
\left| \frac{t - t_i}{t_j - t_i} \sum\limits_{s = t_i}^{t_j - 1} y_s - \sum\limits_{s = t_i}^t y_s - \lambda \right| < \lambda, \quad t = t_i, \dots, t_j - 1,
\end{align*}
\ie
\begin{align} \label{eq:inconsistence1}
-2 \lambda < \sum\limits_{s = t_i}^t y_s - \frac{t - t_i}{t_j - t_i} \sum\limits_{s = t_i}^{t_j - 1} y_s < 0, \quad t = t_i, \dots, t_j - 1.
\end{align}
% (1)
Now, $y_t = e_t + m_t^0$, hence
\begin{align} \label{eq:inconsistence2}
&\sum\limits_{s = t_i}^t y_s - \frac{t - t_i}{t_j - t_i} \sum\limits_{s = t_i}^{t_j - 1} y_s \nonumber \\%
&= \sum\limits_{s = t_i}^t e_s - \frac{t - t_i}{t_j - t_i} \sum\limits_{s = t_i}^{t_j - 1} e_s + \sum\limits_{s = t_i}^t m_s^0 - \frac{t - t_i}{t_j - t_i} \sum\limits_{s = t_i}^{t_j - 1} m_s^0 \nonumber \\
&= \sum\limits_{s = t_i}^t e_s - \frac{t - t_i}{t_j - t_i} \sum\limits_{s = t_i}^{t_j - 1} e_s + \sum\limits_{s = t_i}^t \left( m_s^0 - \frac{1}{t_j - t_i} \sum\limits_{r = t_i}^{t_j - 1} m_r^0 \right) \\
&= \sum\limits_{s = t_i}^t e_s - \frac{t - t_i}{t_j - t_i} \sum\limits_{s = t_i}^{t_j - 1} e_s + \sum\limits_{s = t_i}^t \left[ [m_s^0 - m_{t_k^0}^0] - \sum\limits_{r = t_i}^{t_j - 1} \frac{m_r^0 - m_{t_k^0}^0}{t_j - t_i} \right] \nonumber \\
&> \sum\limits_{s = t_i}^t e_s - \frac{t - t_i}{t_j - t_i} \sum\limits_{s = t_i}^{t_j - 1} e_s, \quad t = t_i, \dots, t_j - 1, \nonumber
\end{align}
% (2)
since $m_s^0 - m_{t_k^0}^0 > 0$ for $t_i \leqslant s < t_k^0$, $m_s^0 - m_{t_k^0}^0 = 0$ for $t_k^0 \leqslant s < t_{k + 1}^0$ and $m_s^0 - m_{t_k^0}^0 < 0$ for $t_{k + 1}^0 \leqslant s < {t_j}$ (\ie $\sum\nolimits_{s = t_i}^t m_s^0 - m_{t_k^0}^0$ is strictly concave, hence its graph lies above the chord determined by the points $(t_i, 0)$ and $(t_j, \sum\nolimits_{s = t_i}^{t_j} m_s^0 - m_{t_k^0}^0)$; this implies that the second term in the fourth line of \eqref{eq:inconsistence2} is strictly positive). Now, denote by $R(t_i, t_j)$ the event that $\sum\nolimits_{s = t_i}^t e_s - \frac{t - t_i}{t_j - t_i} \sum\nolimits_{s = t_i}^{t_j - 1} e_s < 0$ for all $t_i < t < t_k^0 + \varepsilon N/2$ and $t_{k + 1}^0 - \varepsilon N/2 < t < t_j$. By Lemma~\ref{lem:inconsistency2}, given \eg $\delta  = 1$, $P\{ R(t_i, t_j) \}  > 2/(2 \pi^2 N\varepsilon)$ for all $N$ sufficiently large and $\varepsilon$ sufficiently small. Therefore, the probability that \eqref{eq:inconsistence1} holds is upper bounded by
\begin{align*}
&P\left\{ \left. \sum\limits_{s = t_i}^t e_s - \frac{t - t_i}{t_j - t_i} \sum\limits_{s = t_i}^{t_j - 1} e_s < 0 \; \forall t_i < t < t_j \right| Q(t_i, t_j) \right\} \\\
&\qquad \qquad \qquad \qquad \qquad = \frac{P\left\{ \sum\limits_{s = t_i}^t e_s - \frac{t - t_i}{t_j - t_i} \sum\limits_{s = t_i}^{t_j - 1} e_s < 0 \; \forall t_i < t < t_j \right\}}{P\{ R(t_i, t_j)\}} \\
&\qquad \qquad \qquad \qquad \qquad < \frac{\pi^2 N \varepsilon}{t_j - t_i + 1} \\
&\qquad \qquad \qquad \qquad \qquad < \frac{\pi^2 N \varepsilon}{(c - \varepsilon) N} \\
&\qquad \qquad \qquad \qquad \qquad = \frac{\pi^2 \varepsilon}{c - \varepsilon},
\end{align*}
where Lemma~\ref{lem:inconsistency1} has been used. Therefore, taking expectations over $Q(t_i, t_j)$ we see that for $1 - \gamma  \geqslant \pi^2 \varepsilon / (c - \varepsilon)$ the statement of the theorem holds.
\end{proof}

% ----------------------------------------------------
\subsection{Conditions for $\varepsilon$-sign-consistency} \label{subsec:consistency}
The characterization provided by Lemma \ref{lem:3} can be used to establish conditions for the (almost) sparse support recovery of the optimal solution of problem~\eqref{eq:7}. As in the proof of Theorem~\ref{thm:inconsistency}, the slack provided by the definition of $\varepsilon$-sign-consistency needs some special attention; in particular, to show that the dual variables $z_t$ do not cross the $\pm \lambda$ boundary within a given segment, a bound on the min-max value of a sample average has been developed (\cf Lemma~\ref{lem:B2}), which seems to be of independent interest.

\begin{thm}[$\varepsilon$-sign-consistency] \label{thm:4}
Consider the notation of Lemma~\ref{lem:3}, and assume that the data satisfies an equation of the form $\ve{y}_N = \ve{m}_N^0 + \ve{\varepsilon}_N$, where $\ve{\varepsilon}_N$ is a random vector of independent sub-exponential continuous components such that $\kappa := \sup _{N \in \mathbb{N}} \max_{i \in \{1, \ldots, N\}} \| (\ve{\varepsilon}_N)_i\|_{\psi_1} < \infty$ (see Appendix~\ref{app:B} for definitions of these quantities). Furthermore, suppose that the following conditions are satisfied\footnote{The superscript $0$ is used to denote the ``true'' values of $M$, $t_1, \dots, t_M$, and so on.}:
\begin{itemize}
\item[(a)] $M^0 \leqslant M_1$.

\item[(b)] $\min\limits_{k = 2, \ldots, M^0 + 1} (t_k^0 - t_{k - 1}^0) \geqslant M_2 N$.

\item[(c)]  $0 < M_3 \leqslant \min _{k = 2, \ldots, M^0} |m_{t_k^0}^0 - m_{t_k^0 - 1}^0| \leqslant M_4$.

\item[(d)] $\lambda_N = M_5 N^{c_1}$ for some $1/2 < c_1 < 1$.

\item[(e)] All consecutive signs $s_k$ are different, \ie $s_k = -s_{k + 1}$ for all $k = 1, \ldots, M^0 - 2$.
\end{itemize}
Then, the optimal solution $\ve{m}_N$ of problem~\eqref{eq:7} with $\lambda = \lambda_N$ is $\varepsilon$-sign-consistent, \ie it satisfies
\begin{itemize}
\item $(\ve{m}_N)_{(K_N^\varepsilon )^C} = \ve{0}$,

\item For every $i \in K_N^0 := \supp(\ve{m}_N^0)$, there is a $k \in I_i^{\varepsilon ,N}$ such that $\sgn(\ve{m}_N)_k = \sgn(\ve{m}_N^0)_i$,
\end{itemize}
where $I_l^{\varepsilon, N} := \{k \in \{1, \ldots, N - 1\}:\; |k - l| < \varepsilon N\}$ for $i \in \{1, \ldots, N - 1\}$, and $K_N^\varepsilon := \bigcup\nolimits_{i \in K_N^0} I_i^{\varepsilon, N}$.
\end{thm}

\begin{proof}
To simplify the presentation of the proof, consider as an example the sketch in Figure~\ref{fig:1}.

\begin{figure}
\begin{center}
\includegraphics[height=0.4\columnwidth]{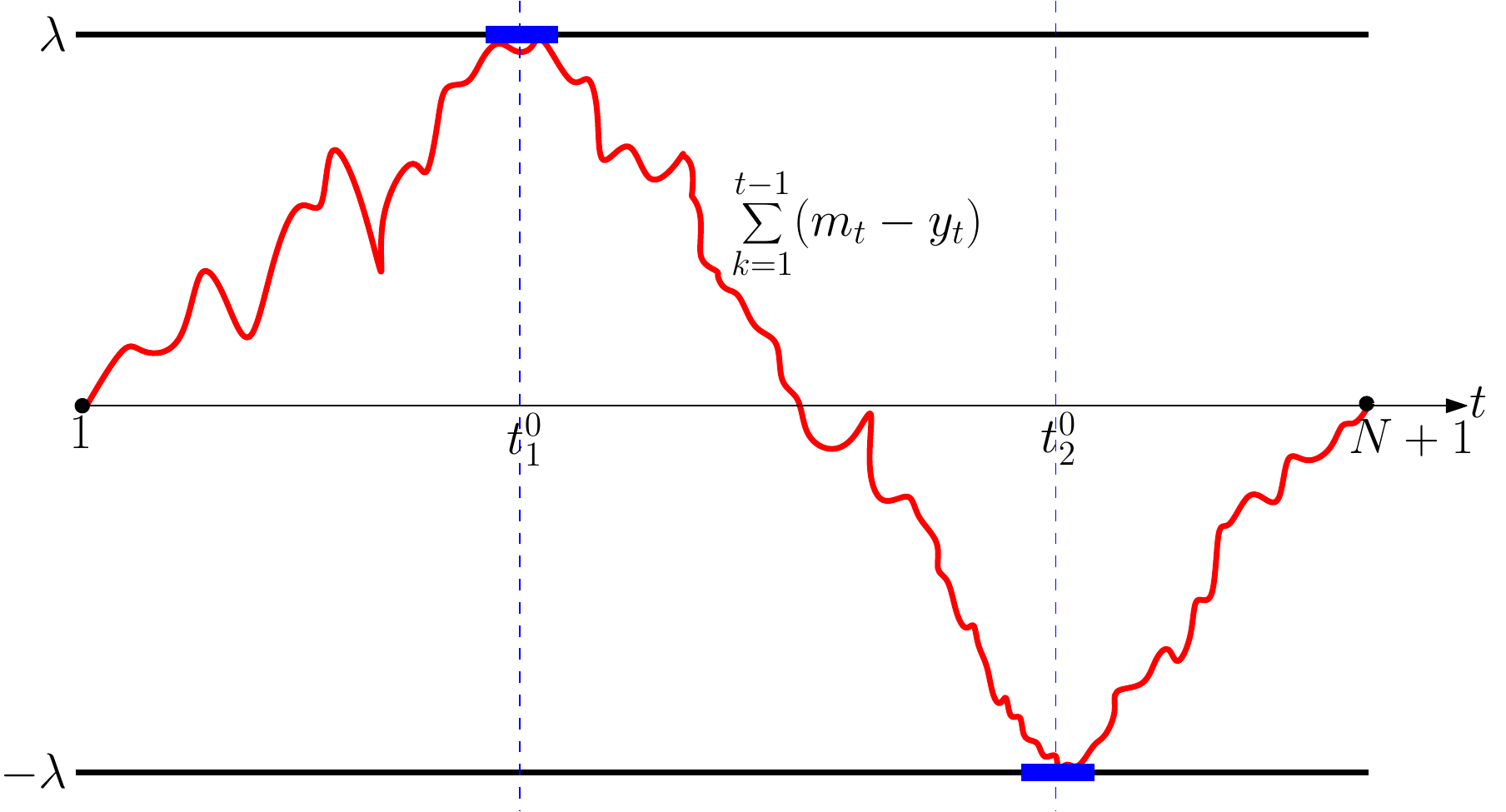}
\end{center}
\caption{Sketch of $z_t = \sum\nolimits_{k = 1}^{t-1} (m_t - y_t)$ (red). According to the optimality conditions from Lemma~\ref{lem:3}, this curve lies between the levels $\lambda$ and $-\lambda$ (black), starting and ending at $0$. The true change points are located at $t = t_1^0$ and $t = t_2^0$. $\varepsilon$-sign consistency means that the curve touches the levels $\lambda$ and $-\lambda$ within an $\varepsilon$-neighborhood of $t_1^0$ and $t_2^0$ (blue bars). }
\label{fig:1}
\end{figure}

Our first goal is to show that, with the given choice of $\lambda$, a sequence $\{m_t\}$, constant along the intervals of $(K_N^\varepsilon)^C$, can be chosen such that the random walk $\{\sum\nolimits_{k = 1}^{t-1} (m_k - y_k)\}$ touches the levels $\pm \lambda$ at the $\varepsilon$-neighborhoods of the true change points, $I_{t_k^0}^{\varepsilon, N}$, where the sign of the level being reached at each $t_k^0$ satisfies the second condition in \eqref{eq:10}. We will require additionally that $\{\sum\nolimits_{k = 1}^{t-1} (m_k - y_k)\}$ reaches $\pm \lambda$ at least once before $t_k^0$ and after $t_k^0$ within $I_{t_k^0}^{\varepsilon, N}$. Let us denote this event as $A$. If, for each $t_k^0$ we denote by $t_k^0 - \delta_k^{(1)}$ and   $t_k^0 + \delta_k^{(2)}$ the first and last touching instants within $I_{t_k^0}^{\varepsilon, N}$ (where, of course, $|\delta_k^{(1)}|, |\delta_k^{(2)}| < \varepsilon$), then, according to Lemma~\ref{lem:4}, this event is equivalent to
\begin{multline} \label{eq:25}
\sgn\left(\frac{1}{t_{k + 1}^0 - \delta_{k + 1}^{(1)} - t_k^0 - \delta _k^{(2)}} \left[ \sum\limits_{t = t_k^0 + \delta_k^{(2)}}^{t_{k + 1}^0 - \delta_{k + 1}^{(1)} - 1} y_t + \lambda (s_{k + 1} - s_k) \right] \right. \\
\left. - \frac{1}{t_k^0 - \delta_k^{(1)} - t_{k - 1}^0 - \delta_{k - 1}^{(2)}} \left[ \sum\limits_{t = t_{k - 1}^0 + \delta_{k - 1}^{(2)}}^{t_k^0 - \delta_k^{(1)} - 1} y_t + \lambda (s_k - s_{k - 1}) \right] \right) = s_k,
\end{multline}
%(25)
where $s_k := \sgn(m_{t_k^0} - m_{t_{k - 1}^0})$. We can assume without loss of generality that $s_k = (-1)^{k + 1}$, and take into account the nature of the data, with which \eqref{eq:25} can be rewritten as
\begin{align*}
&\sgn\left( \left[ \frac{1}{t_{k + 1}^0 - t_k^0 - \delta_{k + 1}^{(1)} - \delta_k^{(2)}} + \frac{1}{t_k^0 - t_{k - 1}^0 - \delta_k^{(1)} - \delta_{k - 1}^{(2)}} \right] 2 \lambda ( - 1)^k + \right .\\
&\frac{1}{t_{k + 1}^0 - \delta_{k + 1}^{(1)} - t_k^0 - \delta_k^{(2)}} \sum\limits_{t = t_k^0 + \delta_k^{(2)}}^{t_{k + 1}^0 - \delta_{k + 1}^{(1)} - 1} \varepsilon_t - \frac{1}{t_k^0 - \delta_k^{(1)} - t_{k - 1}^0 - \delta_{k - 1}^{(2)}} \sum\limits_{t = t_{k - 1}^0 + \delta_{k - 1}^{(2)}}^{t_k^0 - \delta_k^{(1)} - 1} \varepsilon_t \\
&\hspace{8cm}+ m_{t_k^0}^0 - m_{t_{k - 1}^0}^0 \Bigg) = (-1)^{k + 1}.
\end{align*}
In order for the condition to hold, it is sufficient to require that
\begin{align*}
&\left| 2 \lambda \left[ \frac{1}{t_{k + 1}^0 - t_k^0 - \delta_{k + 1}^{(1)} - \delta_k^{(2)}} + \frac{1}{t_k^0 - t_{k - 1}^0 - \delta_k^{(1)} - \delta_{k - 1}^{(2)}} \right] + \right. \\
&\left. \frac{1}{t_{k + 1}^0 - t_k^0 - \delta_{k + 1}^{(1)} - \delta_k^{(2)}} \sum\limits_{t = t_k^0 + \delta_k^{(2)}}^{t_{k + 1}^0 - \delta_{k + 1}^{(1)} - 1} \varepsilon_t - \frac{1}{t_k^0 - t_{k - 1}^0 - \delta_k^{(1)} - \delta_{k - 1}^{(2)}} \sum\limits_{t = t_{k - 1}^0 + \delta_{k - 1}^{(2)}}^{t_k^0 - \delta_k^{(1)} - 1} \varepsilon_t  \right| \\
&\hspace{10cm} \leqslant |m_{t_k^0}^0 - m_{t_{k - 1}^0}^0|,
\end{align*}
and, due to the assumptions of the theorem, it is enough to require that
\begin{align*}
&\left| \frac{1}{t_{k + 1}^0 - t_k^0 - \delta_{k + 1}^{(1)} - \delta_k^{(2)}}\sum\limits_{t = t_k^0 + \delta_k^{(2)}}^{t_{k + 1}^0 - \delta_{k + 1}^{(1)} - 1} \varepsilon_t - \frac{1}{t_k^0 - t_{k - 1}^0 - \delta_k^{(1)} - \delta_{k - 1}^{(2)}} \sum\limits_{t = t_{k - 1}^0 + \delta_{k - 1}^{(2)}}^{t_k^0 - \delta_k^{(1)} - 1} \varepsilon _t \right| \\
&\hspace{8.2cm} \leqslant M_4 - \frac{4 M_5}{(M_2 - 2 \varepsilon)} N^{c_1 - 1}.
\end{align*}
Therefore, the probability of event $A$ is bounded by
\begin{align*}
P\{A\} %
&\geqslant P\left\{ \left| \frac{1}{t_{k + 1}^0 - t_k^0 - \delta_{k + 1}^{(1)} - \delta_k^{(2)}} \sum\limits_{t = t_k^0 + \delta_k^{(2)}}^{t_{k + 1}^0 - \delta_{k + 1}^{(1)} - 1} \varepsilon_t \right. \right. \\
&\qquad \left. \left. - \frac{1}{t_k^0 - t_{k - 1}^0 - \delta_k^{(1)} - \delta_{k - 1}^{(2)}} \sum\limits_{t = t_{k - 1}^0 + \delta_{k - 1}^{(2)}}^{t_k^0 - \delta_k^{(1)} - 1} \varepsilon _t \right| \leqslant M_4 - \frac{4 M_5}{(M_2 - 2 \varepsilon)} N^{c_1 - 1}, \right. \\
&\hspace{8cm} \text{ for all } k = 1, \ldots, M^0 \Bigg\} \\
&\geqslant 1 - 2 M_1 \exp \left[ -c(M_2 - 2 \varepsilon) N \min \left( \frac{1}{\kappa^2} \left( M_4 - \frac{4 M_5}{(M_2 - 2 \varepsilon)} N^{c_1 - 1} \right)^2, \right. \right. \\
&\hspace{6.2cm} \left. \left. \frac{1}{\kappa} \left( M_4 - \frac{4 M_5}{(M_2 - 2 \varepsilon)} N^{c_1 - 1} \right) \right) \right],
\end{align*}
where in the last step we used Lemma~\ref{lem:B1}, and $c > 0$ is an absolute constant.

The next step is to show that it is possible to select $\{m_t\}$ satisfying $A$ and the condition that   $\{\sum\nolimits_{k = 1}^{t-1} (m_k -y_k)\}$ does not reach the levels $\pm$ in $(K_N^\varepsilon)^C$. Let us denote this latter event as $B$. To compute the probability of $B$, consider a given intermediate segment $\{ t_{k - 1}^0, \ldots, t_k^0\}$ where $2 \leqslant k \leqslant M^0 - 1$, for example, the segment $\{t_1^0, \ldots, t_2^0\}$ from Figure~\ref{fig:1}. The event that the random walk touches at least one of the levels $\pm$ in the segment $\{t_{k - 1}^0 + \varepsilon N, \ldots, t_k^0 - \varepsilon N\}$ can be decomposed in the events of separately reaching $\lambda$ and $-\lambda$; let us consider the former sub-event (reaching $\lambda$), since the other can be treated similarly. For simplicity, let us assume also that $s_k = \sgn(m_{t_k^0} - m_{t_{k - 1}^0}) = -1$. We have that
\begin{align} \label{eq:26}
&P\left\{ \sum\limits_{i = 1}^{t-1} (m_i - y_i) > \lambda \text{ for some } t_{k - 1}^0 + \varepsilon N \leqslant t \leqslant t_k^0 - \varepsilon N \right\} \\
&= P\left\{ \sum\limits_{i = t_{k - 1}^0}^{t-1} (m_i - y_i) > \lambda - \sum\limits_{i = 1}^{t_{k - 1}^0 - 1} (m_i - y_i) \text{ for some } t_{k - 1}^0 + \varepsilon N \leqslant t \leqslant t_k^0 - \varepsilon N \right\} \nonumber \\
&\leqslant P\left\{ \sum\limits_{i = t_{k - 1}^0}^{t-1} (m_i - y_i) > 0 \text{ for some } t_{k - 1}^0 + \varepsilon N \leqslant t \leqslant t_k^0 - \varepsilon N \right\}. \nonumber
\end{align}
%(26)
Using Lemma~\ref{lem:5}, \eqref{eq:26} can be upper bounded as follows:
\begin{align} \label{eq:27}
&P\left\{ \sum\limits_{i = t_{k - 1}^0}^{t-1} (m_i - y_i) > 0 \text{ for some } t_{k - 1}^0 + \varepsilon N \leqslant t \leqslant t_k^0 - \varepsilon N \right\} \\
&= P\left\{ \sum\limits_{i = t_{k - 1}^0}^{t-1} \left( \frac{1}{t_k^0 - \delta_k^{(1)} - t_{k - 1}^0 + \delta_{k - 1}^{(2)}} \left[ \sum\limits_{j = t_{k - 1}^0 + \delta_{k - 1}^{(2)}}^{t_k^0 - \delta_k^{(1)} - 1} (m_j^0 + \varepsilon_j)  - 2\lambda \right] - m_i^0 - \varepsilon_i \right) \right. \nonumber \\
&\hspace{5.5cm} \left. > 0 \text{ for some } t_{k - 1}^0 + \varepsilon N \leqslant t \leqslant t_k^0 - \varepsilon N \rule{0cm}{0.85cm} \right\} \nonumber \\
&\leqslant P\left\{ \frac{1}{t - t_{k - 1}^0} \sum\limits_{i = t_{k - 1}^0}^{t-1} \left( \tilde{\varepsilon}_i - \frac{1}{t_k^0 - \delta_k^{(1)} - t_{k - 1}^0 + \delta_{k - 1}^{(2)}} \sum\limits_{j = t_{k - 1}^0 + \delta_{k - 1}^{(2)}}^{t_k^0 - \delta_k^{(1)} - 1} \varepsilon_j \right) > \right. \nonumber \\
&\hspace{2.5cm} \left. \frac{2M_5 N^{c_1}}{t_k^0 - \delta_k^{(1)} - t_{k - 1}^0 + \delta_{k - 1}^{(2)}} \text{ for some } t_{k - 1}^0 + \varepsilon N \leqslant t \leqslant t_k^0 - \varepsilon N \rule{0cm}{0.85cm} \right\} \nonumber \\
&= P\left\{ \max\limits_{\varepsilon N \leqslant t \leqslant t_k^0 - t_{k - 1}^0 - \varepsilon N} \frac{1}{t} \sum\limits_{i = 0}^{t-1} \left( \tilde{\varepsilon}_{i + t_{k - 1}^0} - \frac{1}{t_k^0 - \delta_k^{(1)} - t_{k - 1}^0 + \delta_{k - 1}^{(2)}} \sum\limits_{j = t_{k - 1}^0 + \delta_{k - 1}^{(2)}}^{t_k^0 - \delta_k^{(1)} - 1} \varepsilon_j \right) \right. \nonumber \\
&\hspace{7.7cm} \left. > \frac{2 M_5 N^{c_1}}{t_k^0 - \delta_k^{(1)} - t_{k - 1}^0 + \delta_{k - 1}^{(2)}} \rule{0cm}{0.85cm} \right\}, \nonumber
\end{align}
%(27)
where $\tilde{\varepsilon}_t := -\varepsilon_t$. From Corollary~\ref{cor:B1},
\begin{align*}
&P\left\{ \max\limits_{\varepsilon N \leqslant t \leqslant t_k^0 - t_{k - 1}^0 - \varepsilon N} \frac{1}{t} \sum\limits_{i = 0}^{t-1} \left( \tilde{\varepsilon}_{i + t_{k - 1}^0} - \frac{1}{t_k^0 - \delta_k^{(1)} - t_{k - 1}^0 + \delta_{k - 1}^{(2)}} \sum\limits_{j = t_{k - 1}^0 + \delta_{k - 1}^{(2)}}^{t_k^0 - \delta_k^{(1)} - 1} \varepsilon_j \right) \right. \\
&\hspace{8cm} \left. > \frac{2 M_5 N^{c_1}}{t_k^0 - \delta_k^{(1)} - t_{k - 1}^0 + \delta_{k - 1}^{(2)}} \rule{0cm}{0.85cm} \right\}
\end{align*}
\begin{align*}
&\leqslant P\left\{ \max\limits_{\delta_{k - 1}^{(2)} \leqslant t \leqslant N} \frac{1}{t} \sum\limits_{i = 0}^{t-1} \left( \tilde{\varepsilon}_{i + t_{k - 1}^0} - \frac{1}{t_k^0 - \delta_k^{(1)} - t_{k - 1}^0 + \delta_{k - 1}^{(2)}} \sum\limits_{j = t_{k - 1}^0 + \delta_{k - 1}^{(2)}}^{t_k^0 - \delta_k^{(1)} - 1} \varepsilon_j \right) \right. \\
&\hspace{10cm}\left. > 2 M_5 N^{c_1 - 1} \rule{0cm}{0.85cm} \right\} \\
&\leqslant 2[(1 - \varepsilon) N + 1]\exp \left[ -\frac{c \varepsilon}{1 - \varepsilon} \min \left( \frac{4 M_5^2 N^{2 c_1 - 1}}{\kappa^2}, \frac{2 M_5 N^{c_1}}{\kappa} \right) \right],
\end{align*}
where $c > 0$ is an absolute constant. In the case of the initial segment, $\{1, \ldots, t_1^0\}$, a similar calculation gives
\begin{multline*}
P\left\{ \sum\limits_{i = 0}^{t-1} (m_i - y_i)  > \lambda \text{ for some } 1 \leqslant t \leqslant t_1^0 - \varepsilon N \right\} \\
\leqslant 2 [(1 - \varepsilon) N + 1] \exp \left[ -\frac{c \varepsilon}{1 - \varepsilon} \min \left( \frac{M_5^2 N^{2 c_1 - 1}}{\kappa^2}, \frac{M_5 N^{c_1}}{\kappa} \right) \right],
\end{multline*}
and an analogous computation provides a bound for the final segment $\{t_{M^0 - 1}^0,$ $\ldots, N\} $.

Combining all the previous results (for the full set of segments $\{t_{k - 1}^0, \ldots, t_k^0\}$) gives the following bound for the probability of event $B$:
\begin{align*}
P\{ B\} %
&\geqslant 1 - 8[(1 - \varepsilon) N + 1] \exp \left[ -\frac{c \varepsilon}{1 - \varepsilon} \min \left( \frac{M_5^2 N^{2 c_1 - 1}}{\kappa^2}, \frac{M_5 N^{c_1}}{\kappa} \right) \right] \hfill \\
&\quad -4 (M^0 - 1) [(1 - \varepsilon) N + 1] \exp \left[ -\frac{c \varepsilon}{1 - \varepsilon} \min \left( \frac{4 M_5^2 N^{2 c_1 - 1}}{\kappa^2}, \frac{2 M_5 N^{c_1}}{\kappa} \right) \right] \\
&\geqslant 1 - 8[(1 - \varepsilon) N + 1] \exp \left[ -\frac{c \varepsilon}{1 - \varepsilon} \min \left( \frac{M_5^2 N^{2 c_1 - 1}}{\kappa^2}, \frac{M_5 N^{c_1}}{\kappa} \right) \right] \\
&\quad - 4 (M_1 - 1) [(1 - \varepsilon) N + 1] \exp \left[ -\frac{c \varepsilon}{1 - \varepsilon} \min \left( \frac{4 M_5^2 N^{2 c_1 - 1}}{\kappa^2}, \frac{2 M_5 N^{c_1}}{\kappa} \right) \right].
\end{align*}
Therefore, $P\{ A\}, P\{ B\} \xrightarrow{N \to \infty} 1$, which establishes the $\varepsilon$-sign consistency of the solution of problem~\eqref{eq:7}. This concludes the proof.
\end{proof}

\begin{cor}[Rate of convergence of $\varepsilon$-sign-consistency] \label{cor:1}
Under the conditions of Theorem~\ref{thm:4}, the probability $P_N$ that the FLSA fails to recover the true support of $\ve{m}_N^o$ is dominated by $\exp(-N^{2 c_1 - 1})$ (recall that $1/2 < c_1 < 1$ is the exponent of $\lambda$), in the sense that there exists a $C > 0$ such that
\begin{align*}
\lim_{N \to \infty} \frac{-\ln P_N}{N^{2 c_1 - 1}} \leqslant C.
\end{align*}
\end{cor}

This corollary implies that the choice of $\lambda$ ultimately determines the rate of convergence of the change points of the FLSA estimate to the true change points of $\ve{m}_N^o$.

\begin{remm}
The results of Sections~\ref{subsec:inconsistency} and \ref{subsec:consistency} basically show that the FLSA fails to detect the change points of the underlying signal only in the presence of ``staircases''. This phenomenon may be detected by looking at the dual variables, as the examples in the Introduction suggest, since a staircase forces the dual variables $z_t$ to remain close to the boundaries $\pm \lambda$. Therefore, by observing $z_t$ it is possible to distinguish in a first stage between legitimate change points and those which may be mere staircase artifacts.
\end{remm}

%{\color{red}
%{\bf Bo:} We need to follow the advice of M. Jordan to include a corollary with the rate of converge. Is the answer
%$ N^{2c_1-1}$ since it is smaller than $N^{c_1}$ for $1/2\leq c_1<1$
%
%Hence, the asymptotic size of $\lambda$ is the most important parameter?
%
%}

% ==============================
\section{Extensions}
\label{sec:3}

\subsection{Mean and variance filtering}
Consider the signal $\{y_t\}$ which satisfies $y_t \sim \mathcal{N}(m_t, \sigma_t^2)$, where both $\{m_t\}$ and $\sigma_t$ are (unknown) piece-wise constant sequences. Assume that the measurements $\ve{Y}_N := [y_1 \; \cdots y_N]^T$ are available, and we are interested in estimating $m_1, \dots, m_N$ and $\sigma_1, \dots, \sigma_N$.

To solve this problem, first notice that the model $y_t \sim \mathcal{N}(m_t, \sigma_t^2)$ is a \emph{standard exponential family with canonical parameters} $\mu_t := m_t / \sigma_t^2$ and $\eta_t := -1 / 2\sigma_t^2$, where $\mu_t \in \mathbb{R}$ and $\eta_t \in \mathbb{R}^{-}$ \citep{Brown-86}[Example~1.2]. This means that the log-likelihood of $\{\mu_1, \dots, \mu_N, \eta_1, \dots, \eta_N\}$  given $\ve{Y}_N$ is
\begin{align*}
&l(\mu_1, \ldots, \mu_N, \eta_1, \ldots, \eta_N) \\%
&\qquad \qquad = \ln \left\{ \frac{1}{(2\pi)^{N/2} \prod\nolimits_{t = 1}^N \sigma_t }\exp \left(-\sum\limits_{t = 1}^N \frac{(y_t - m_t)^2}{2 \sigma_t^2} \right) \right\} \\
&\qquad \qquad = - \frac{N}{2} \ln(2 \pi) - \sum\limits_{t = 1}^N \ln(\sigma_t) - \sum\limits_{t = 1}^N \frac{(y_t - m_t)^2}{2\sigma_t^2} \\
&\qquad \qquad = - \frac{N}{2} \ln \pi + \frac{1}{2} \sum\limits_{t = 1}^N \ln(-\eta_t) + \sum\limits_{t = 1}^N \left[ \eta_t y_t^2 + \mu_t y_t + \frac{\mu_t^2}{4\eta_t} \right] \\
&\qquad \qquad = - \frac{N}{2} \ln \pi + \frac{1}{2} \sum\limits_{t = 1}^N \ln (-\eta_t)  + \sum\limits_{t = 1}^N \frac{\mu_t^2}{4 \eta_t} + \sum\limits_{t = 1}^N (\eta _t y_t^2 + \mu_t y_t).
\end{align*}
Moreover, by \citep{Brown-86}[Theorem~1.13] it follows that $l$ is strictly concave on $\{(\mu_1, \ldots,$ $\mu_N, \eta_1, \ldots, \eta_N):\mu_t \in \mathbb{R}, \eta_t \in \mathbb{R}^-, t = 1, \ldots, N\} $. In order to impose the prior knowledge on the piece-wise constant character of $\{m_t\}$ and $\{\sigma_t\}$, we propose (inspired by \citep{Kim-Koh-Boyd-Gorinevsky-09}) an estimator based on the solution of the following optimization problem:
\begin{align} \label{eq:1}
\begin{array}{cl}
\min\limits_{\scriptstyle \mu_1, \ldots, \mu_N, \eta_1, \ldots , \eta_N} & - \displaystyle \sum\limits_{t = 1}^N \left( \frac{1}{2} \ln(-\eta_t) + \frac{\mu_t^2}{4 \eta_t} +\eta_t y_t^2 + \mu_t y_t \right) \\ &\qquad\qquad
+ \displaystyle \sum\limits_{t = 2}^N \left( \lambda_1 |\mu _t - \mu_{t - 1}| + \lambda_2 |\eta_t - \eta_{t - 1}| \right) \\
\text{s.t.} & \eta_t < 0, \quad t = 1, \ldots, N.
\end{array}
\end{align}
% (1)

Let us consider now the variance-only case, \ie where $m_1 = \dots = m_N = 0$. Under this assumption, \eqref{eq:1} can be written as
\begin{align} \label{eq:2}
\begin{array}{cl}
\min\limits_{\eta_1, \ldots , \eta_N} &  - \frac{1}{2} \sum\limits_{t = 1}^N \ln(-\eta_t) - \sum\limits_{t = 1}^N \eta_t y_t^2 + \lambda \sum\limits_{t = 2}^N \left| \eta_t - \eta_{t - 1} \right| \\
\text{s.t.} & \eta_t < 0, \quad t = 1, \ldots, N,
\end{array}
\end{align}
% (2)
where we have dropped the subscript of $\lambda_2$ to simplify the notation.

%Our first result establishes the existence of a $\lambda_{\max} > 0$  such that for all $\lambda \geqslant \lambda_{\max}$, the optimal solution of \eqref{eq:2} is constant.

%\begin{thm}[Maximum value of $\lambda$] \label{thm:1}
%Let $\eta^\text{const}$ be the optimal constant solution of problem \eqref{eq:2}, \ie
%
%\begin{align} \label{eq:3}
%\eta^\text{const} %
%&= \arg\min\limits_{\eta < 0} \left(-\frac{1}{2} \sum\limits_{t = 1}^N \ln(-\eta) - \sum\limits_{t = 1}^N \eta y_t^2 \right) \nonumber \\
%&= \arg\min\limits_{\eta < 0} \left(-\frac{N}{2} \ln(-\eta) - \eta \sum\limits_{t = 1}^N y_t^2 \right) \\
%&= -\frac{N}{2\sum\limits_{t = 1}^N y_t^2}, \nonumber
%\end{align}
%(3)
%and let
%
%\begin{align*}
%\lambda_{\max} = \max\limits_{t = 1, \ldots, N - 1} \left| \sum\limits_{k = 1}^t \left( \frac{1}{2 \eta^\text{const}} + y_k^2 \right) \right| = \max\limits_{t = %1, \ldots, N - 1} \left| \sum\limits_{k = 1}^t y_k^2 - \frac{t}{N}\sum\limits_{k = 1}^N y_k^2 \right|.
%\end{align*}
%
%Then, $\eta_1 = \cdots = \eta_N = \eta^\text{const}$ is the optimal solution of \eqref{eq:2} iff $\lambda \geqslant \lambda_{\max}$.
%\end{thm}

The KKT conditions of the optimal solution of \eqref{eq:2} are given next.

\begin{lem}[KKT conditions] \label{lem:1}
The KKT conditions of \eqref{eq:2} are
\begin{align} \label{eq:4}
\ve{y} - \ve{\sigma} = \ve{A} \tilde{\ve{z}},\quad \ve{\sigma} > 0
\end{align}
%(4)
where $\ve{y} := [y_1^2 \; \cdots \; y_N^2]^T$, $\ve{\sigma} := [-1/2 \eta_1 \; \cdots \; -1 / 2 \eta_N]^T = [\sigma_1^2 \; \cdots \; \sigma_N^2]^T$, and
\begin{align} \label{eq:5}
&\ve{A} := \left[ \begin{array}{cccc}
1  &            &           & 0  \\
-1 &1          &           &     \\
    & \ddots & \ddots &    \\
    &            & -1       & 1  \\
0  &            &           & -1
\end{array} \right], \quad
\tilde{\ve{z}} := \left[ \begin{array}{c}
\tilde{z}_2 \\
\vdots \\
\tilde{z}_N
\end{array} \right], \\
&\tilde{z}_t \left\{ \begin{array}{ll}
= \lambda \sgn (\eta_t - \eta_{t - 1}), & \text{if } \eta_t \ne \eta_{t - 1} \\
 \, \in \lambda [-1, 1], & \text{otherwise.}
\end{array} \right. \nonumber
\end{align}
%(5)
\end{lem}

An important observation from Lemma~\ref{lem:1} is that the KKT conditions for the solution of problem \eqref{eq:2} coincide with those of the so-called fused lasso (or, more precisely, the FLSA) \citep{Friedman-Hastie-Hoffling-Tibshirani-07}. This is formally established in the following lemma, originally proved in \cite{BW_Asilomar}.

\begin{lem}[Relation to fused lasso] \label{lem:2}
The solution of problem \eqref{eq:2} coincides with the FLSA, given by the solution of the optimization problem:
\begin{align} \label{eq:6}
\begin{array}{cl}
\min\limits_{\sigma_1^2, \ldots, \sigma_N^2} & \displaystyle \frac{1}{2} \displaystyle \sum\limits_{t = 1}^N [y_t^2 - \sigma_t^2]^2 + \lambda \sum\limits_{t = 2}^N \left| \sigma_t^2 - \sigma_{t - 1}^2 \right|.
\end{array}
\end{align}
%(6)
\end{lem}
The conclusion is that the theory on FLSA directly applies to the corresponding variance segmentation and estimation problem. The corresponding multivariate covariance matrix problems is more difficult to analyze and is outside the scope of the paper.
From a practical point of view it may not be good to square the measurements $y_t^2$, since it amplifies noise and outliers. A solution is to use the more robust Huber penalty function instead of the least squares cost, while still leading to a convex optimization problem.

\subsection{Trend Filtering}
Another question is how to apply the theory of Section  \ref{sec:2} to the trend filtering problem discussed in \cite{Kim-Koh-Boyd-Gorinevsky-09}. Here the mean values should be piece-wise linear. An simple approach would be to apply the FLSA to
$$
\Delta y(t)=y(t)-y(t-1)
$$
which would be piece-wise constant. Hence we would expect that the staircase issue would arise if the slopes of consecutive linear segments are increasing (or decreasing). However, the $\ell_1$ filtering algorithm is a bit more involved: % and corresponds to 
\begin{align}
\begin{array}{cl}
\min\limits_{\{m_t\}_{t=1}^N,\{w_t\}_{t=2}^N} &\; \displaystyle \frac{1}{2}\sum_{t=1}^N [y_t-m_t]^2+\lambda \sum_{t=2}^{N-1} |w_t|\\
\text{s.t.} &\;  w_t=m_{t+1}-2m_t+m_{t-1}, \quad t = 2, \dots, N-1.
\end{array}
\label{eq:trend}
\end{align}

 The KKT optimality conditions can be derived as follows.
Let
\begin{align*}
%\label{eq:lagrtrend}
{\L}(\ve{m},\ve{w},\ve{z})=\frac{1}{2}\sum_{t=1}^N [y_t-m_t]^2+\lambda \sum_{t=2}^{N-1} |w_t|+ \sum_{t=2}^{N-1}z_t(
m_{t+1}-2m_{t}+m_{t-1}-w_{t})
\end{align*}
with respect to $m_t$ to obtain
\begin{align} 
-(y_1-m_1)+z_2 &= 0, \nonumber\\
-(y_2-m_2)-2z_2+z_3 &= 0, \nonumber\\
-(y_t-m_t)+z_{t+1}-2z_{t}+z(t-1) &= 0, \quad t=3,\ldots, N-2,\nonumber\\
-(y_{N-1}-m_{N-1})+z_{N-1}-2z_{N-2} &= 0,  \nonumber\\
-(y_N-m_N)+z_N &= 0. \nonumber
\end{align}
The solution of these equations is just the double sum
$$
z_t=\sum^{t-2}_{j=3}\sum_{i=1}^{j-1} [m_i-y_i]
$$
with proper initial and end constraints given by \eqref{eq:optm1}. The sub-gradient with respect to $w_t$ is exactly the same as for FLSA.
Hence we have an integrated random walk with endpoints at
$
z_t=\lambda \Sgn(m_{t+1}-2m_{t-1}+m_{t-1}),\quad t=2, \ldots, N-1.
$
To conclude:
\begin{align} \label{eq:optcondtrend}
& |  z_{t}|\leqslant \lambda,\; t=2,\ldots , N-1,\nonumber\\
& |  z_{t}|<\lambda\;\mbox{(constant)}\quad \Rightarrow\: m_{t+1}-2m_t+m_{t-1}=0 \\
&| z_{t_k}|=\lambda \;\mbox{(transition)} \quad \Rightarrow\;
 \sgn(m_{t_{k}+1}-2m_{t_{k}}+m_{t_{k}-1}) = \sgn(z_{t_k}). \nonumber
\end{align}
The properties of the corresponding estimates can in principle be analyzed from the bias term of $z_t$. This will, however, be the topic of future research.

% ==============================
\section{Summary} \label{sec:4}

In this paper, the change point detection properties of the fused lasso have been studied.
%This is a simple, but very relevant estimation problem with many free variables.
In contrast to previous results in the literature, which establish the impossibility of the fused lasso to exactly determine the true change points of a piece-wise signal, our analysis has focused on the \emph{approximate} detection of such change points, by defining the concept of $\varepsilon$-sign consistency. As a result, we have shown that the $l_1$ regularization trick of the fused lasso works or fails in detecting the true change points under well defined conditions, based on the intuition obtained from the Lagrangian dual of the FLSA. It is important to notice, however, that the FLSA is $\ell_2$ consistent under milder conditions (given a suitable choice of its regularization parameter).

% ==============================
\appendix

% ==============================
\section{Proofs} \label{app:A}

% -----------------------------------------------------
\subsection{Proof of Lemma~\ref{lem:3}}
The first two equations in \eqref{eq:10} can be obtained by adding the first $k$ ($1 \leqslant k \leqslant N - 1$) components of \eqref{eq:4}. By adding all the components of \eqref{eq:8} we arrive at the third equation of \eqref{eq:10}. The converse can be established by subtracting consecutive components of \eqref{eq:10}. This concludes the proof.	\qed

% -----------------------------------------------------
\subsection{Proof of Lemma~\ref{lem:4}}
From the knowledge of the $t_k$'s and $s_{t_k}$'s (taking $s_{t_M} = s_N := 0$), the cost of problem \eqref{eq:7} can be written as
\begin{align*}
f = \sum\limits_{k = 1}^M \left[ \frac{1}{2} \sum\limits_{t = t_{k - 1}}^{t_k - 1} y_t^2 - x_{t_{k - 1}} \sum\limits_{t = t_{k - 1}}^{t_k - 1} y_t  + \frac{t_k - t_{k - 1}}{2} x_{t_{k - 1}}^2 + \lambda s_k (x_{t_k} - x_{t_{k - 1}}) \right].
\end{align*}
Therefore, by differentiating $f$ with respect to the $x_{t_k}$'s and setting the derivatives to zero, we obtain
\begin{align*}
(t_1 - 1) x_1 - \sum\limits_{t = 1}^{t_1 - 1} y_t - \lambda s_1 &= 0 \\
(t_{k + 1} - t_k) x_{t_k} - \sum\limits_{t = t_k}^{t_{k + 1} - 1} y_t - \lambda s_{k + 1} + \lambda s_k &= 0, \quad k = 1, \ldots, M - 1,
\end{align*}
or
\begin{align*}
x_1 &= \frac{1}{t_1 - 1} \sum\limits_{t = 1}^{t_1 - 1} y_t + \frac{1}{t_1 - 1} \lambda s_1 \\
x_{t_k} &= \frac{1}{t_{k + 1} - t_k} \sum\limits_{t = t_k}^{t_{k + 1} - 1} y_t + \frac{1}{t_{k + 1} - t_k} \lambda (s_{k + 1} - s_k), \quad k = 1, \ldots, M - 1.
\end{align*}
This concludes the proof.	\qed

\subsection{Proof of Lemma~\ref{lem:6}}
By performing the change of variables suggested in the statement of the lemma, \eqref{eq:11} can be put in the form
\begin{align} \label{eq:37}
\min\limits_{x_1, \ve{x'}} \; \frac{1}{2} \left\| \ve{y} - x_1 \ve{1}_N - \left[ \begin{array}{c}
0_{1,N - 1} \\
\ve{A}_{N - 1}
\end{array} \right] \ve{x'} \right\|_2^2 + \lambda \| \ve{x'} \|_1.
\end{align}
%(37)
Since the cost function in \eqref{eq:37} is quadratic in $x_1$, it can be simplified by explicitly minimizing this cost with respect to $x_1$. To this end, notice that
\begin{align*}
&\left\| \ve{y} - x_1 \ve{1}_N - \left[ \begin{array}{c}
0_{1,N - 1}      \\
\ve{A}_{N - 1}
\end{array} \right] \ve{x'} \right\|_2^2 \\
&= \left( \ve{y} - x_1 \ve{1}_N - \left[ \begin{array}{c}
0_{1,N - 1}      \\
\ve{A}_{N - 1}
\end{array} \right] \ve{x'} \right)^T \left( \ve{y} - x_1 \ve{1}_N - \left[ \begin{array}{c}
0_{1,N - 1}      \\
\ve{A}_{N - 1}
\end{array} \right] \ve{x'} \right) \\
&= \left( \ve{y} - \left[ \begin{array}{c}
0_{1,N - 1} \\
\ve{A}_{N - 1}
\end{array} \right] \ve{x'} \right)^T \left[ \ve{I} - \frac{1}{N} \ve{1}_{N,N} \right] \left( \ve{y} - \left[ \begin{array}{c}
0_{1,N - 1}      \\
\ve{A}_{N - 1}
\end{array} \right] \ve{x'} \right) \\
&\qquad + N \left( x_1 - \frac{1}{N} \ve{1}_{1,N} \left[ \ve{y} - \left[ \begin{array}{c}
0_{1,N - 1} \\
\ve{A}_{N - 1}
\end{array} \right] \ve{x'} \right] \right)^2,
\end{align*}
which shows that \eqref{eq:37} can be replaced by
\begin{align} \label{eq:38}
\min\limits_{\ve{x'}} \; \frac{1}{2} \left( \ve{y} - \left[ \begin{array}{c}
0_{1,N - 1} \\
\ve{A}_{N - 1}
\end{array} \right] \ve{x'} \right)^T \left[ \ve{I} - \frac{1}{N} \ve{1}_{N,N} \right] \left( \ve{y} - \left[ \begin{array}{c}
0_{1,N - 1}      \\
\ve{A}_{N - 1}
\end{array} \right] \ve{x'} \right) + \lambda \|\ve{x'}\|_1,
\end{align}
%(38)
and
\begin{align*}
x_1 = \frac{1}{N} \ve{1}_{1,N} \left[ \ve{y} - \left[ \begin{array}{c}
0_{1,N - 1}      \\
\ve{A}_{N - 1}
\end{array} \right] \ve{x'} \right] = \frac{1}{N} \sum\limits_{t = 1}^N y_t  - \frac{1}{N} \sum\limits_{t = 1}^{N - 1} \sum\limits_{k = 1}^t x'_k.
\end{align*}
Furthermore, since $\ve{I} - N^{ - 1} \ve{1}_{N,N}$ is idempotent, and
\begin{multline*}
\left[ \ve{I} - \frac{1}{N} \ve{1}_{N,N} \right] \left( \ve{y} - \left[ \begin{array}{c}
0_{1,N - 1}      \\
\ve{A}_{N - 1}
\end{array} \right] \ve{x'} \right) \\
= \ve{y} - \left( N^{-1} \sum\nolimits_{t = 1}^N y_t \right) \ve{1}_{N,1} - \left[ \ve{I} - \frac{1}{N} \ve{1}_{N,N} \right] \left[ \begin{array}{c}
0_{1,N - 1}      \\
\ve{A}_{N - 1}
\end{array} \right] \ve{x'} = \ve{\tilde{y}} - \ve{\tilde{A} x'},
\end{multline*}
where we have used the notation in the statement of the theorem, we have that \eqref{eq:38} is equal to \eqref{eq:12}. This concludes the proof. \qed

% -----------------------------------------------------
\subsection{Proof of Lemma~\ref{lem:7}}
First notice that $\ve{C}$ is symmetric by construction, which establishes the first equality in \eqref{eq:13}. Now, let $i \leqslant k$. Then, by Lemma~\eqref{eq:6}, we have that
\begin{align*}
C_{ik} %
&= \sum\limits_{l = 1}^N A_{li} A_{lk} \\
&= \sum\limits_{l = 1}^i \frac{i - N}{N} \frac{k - N}{N} + \sum\limits_{l = i + 1}^k \frac{i}{N} \frac{k - N}{N} + \sum\limits_{l = k + 1}^N \frac{i}{N}\frac{k}{N} \\
&= \frac{i (i - N) (k - N) + (k - i) i (k - N) + (N - k) i k}{N^2} \\
&= \frac{i (N - k)}{N}.
\end{align*}
This proves the Lemma. \qed

% -----------------------------------------------------
\subsection{Proof of Lemma~\ref{lem:8}}
To simplify the proof, let us extend $\ve{X}$ to $\ve{\tilde{X}} := \ve{C}_{:,K} \ve{C}_{K,K}^{-1} \in \mathbb{R}^{n \times |K|}$ (where we have use Matlab$^\circledR$'s notation). This is equivalent to stating that $\ve{C}_{K,K} \ve{\tilde{X}}^T = \ve{C}_{K,:}$ (due to the symmetry of $\ve{C}$). We will show that for every $i \in \{1, \ldots, n\}$, $k \in \{1, \ldots, |K|\}$,
\begin{align*}
\ve{\tilde{X}}_{i,k} = \left\{ \begin{array}{ll}
0, & i \leqslant K(k - 1) \\
\displaystyle \frac{i - K(k - 1)}{K(k) - K(k - 1)},   & K(k - 1) \leqslant i \leqslant K(k) \\
\displaystyle \frac{K(k + 1) - i}{K(k + 1) - K(k)}, & K(k) < i \leqslant K(k + 1) \\
0,                                                  & i > K(k + 1).
\end{array} \right.
\end{align*}
To this end, first notice that if $i = K(\tilde{k})$ for some $\tilde{k} \in \{1, \ldots, |K|\}$, then the $i$-th row of $\ve{X}$ equals $\ve{e}_{\tilde{k}}^T$ (the $\tilde{k}$-th unit row vector in $\mathbb{R}^{|K|}$). This is so because $\ve{C}_{K,K} \ve{e}_{\tilde{k}} = \ve{C}_{K,\tilde{k}}$, which corresponds to the $i$-th column of the equation $\ve{C}_{K,K} \ve{\tilde{X}}^T = \ve{C}_{K,:}$. To conclude the proof, it is enough to show that if $i$ lies between, say, $K(\tilde{k})$ and $K(\tilde{k} + 1)$ for some $\tilde{k} \in \{0, \ldots, |K|\}$, the $i$-th row of $\ve{X}$ is a linear interpolation of the rows $\ve{X}_{K(\tilde{k}), :}$ and $\ve{X}_{K(\tilde{k} + 1), :}$, or, equivalently, that $\ve{X}_{i,:}$ depends affinely on $i$ between $K(\tilde{k})$ and $K(\tilde{k} + 1)$. This follows directly from the equation $\ve{C}_{K,K} (\ve{\tilde{X}}_{i,:})^T = \ve{C}_{K,i}$, since for $K(\tilde{k}) \leqslant i \leqslant K(\tilde{k} + 1)$ we have, by Lemma~\ref{lem:1},
\begin{align*}
\ve{C}_{k,i} = \left\{ \begin{array}{ll}
\vspace{2mm} \displaystyle \ve{C}_{k, \tilde{k}} + \frac{k}{N} [K(\tilde{k}) - i], & k \leqslant \tilde{k} \\
\displaystyle \ve{C}_{k, \tilde{k}} + \frac{(N - k)}{N} [i - K(\tilde{k})], & k > \tilde{k}.
\end{array} \right.
\end{align*}
Therefore, since the right hand side of $\ve{C}_{K,K} (\ve{\tilde{X}}_{i,:})^T = \ve{C}_{K,i}$ depends affinely on $i$ for $K(\tilde{k}) \leqslant i \leqslant K(\tilde{k} + 1)$, so does $(\ve{\tilde{X}}_{i,:})^T$. This concludes the proof.	\qed

\subsection{Proof of Lemma~\ref{lem:10}}
Let us assume that condition~\eqref{eq:15} does not hold for a particular $K$, and pick $\ve{x}_0 \in \mathbb{R}^n$ as in the statement of the theorem, \ie $\sgn(\ve{x}_0)_K = \ve{s}$ and $(\ve{x}_0)_{\{1, \ldots, n\} \backslash K} = 0$. Consider the subdifferential of the cost function of \eqref{eq:14}:
\begin{align*}
\partial \left[ \frac{1}{2}\| \ve{y} - \ve{Ax} \|_2^2 + \lambda \| \ve{x} \|_1 \right] = \ve{A}^T (\ve{y} - \ve{Ax}) + \lambda \Sgn(\ve{x}),
\end{align*}
where ``$\Sgn$'' is a set-valued version of $\sgn$, applied component-wisely: $\Sgn(x) = \{1\}$ if $x > 0$, $\Sgn(x) = \{-1\}$ if $x < 0$  and $\Sgn(0) = [-1, 1]$. Now, $\ve{x}$ is an optimal solution of \eqref{eq:14} iff
\begin{align*}
\ve{0} \in \ve{A}^T (\ve{y} - \ve{Ax}) + \lambda \Sgn(\ve{x}),
\end{align*}
\ie
\begin{align*}
\ve{C}(\ve{x}_0 - \ve{x}) + \ve{A}^T \ve{\varepsilon} \in -\lambda \Sgn(\ve{x}).
\end{align*}
Let us assume that $\sgn(\ve{x}) = \sgn(\ve{x}_0)$. This implies that
\begin{align*}
\ve{C}_{K^C,K}(\ve{x}_0 - \ve{x})_K + (\ve{A}^T \ve{\varepsilon})_{K^C} &= -\lambda \ve{w} \\
\ve{C}_{K,K}(\ve{x}_0 - \ve{x})_K + (\ve{A}^T \ve{\varepsilon})_K &= -\lambda \ve{s},
\end{align*}
where $\ve{w} \in [-1, 1]^{n - |K|}$ is arbitrary. Combining these equations we obtain
\begin{align} \label{eq:40}
\ve{C}_{K^C,K} \ve{C}_{K,K}^{-1} \ve{s} = \frac{1}{\lambda} \left[ (\ve{A}^T \ve{\varepsilon})_{K^C} - \ve{C}_{K^C,K} \ve{C}_{K,K}^{-1}(\ve{A}^T \ve{\varepsilon})_K \right] + \ve{w}.
\end{align}
%(40)
Let $i \in K^C$ be such that $|(\ve{C}_{K^C,K} \ve{C}_{K,K}^{-1} \ve{s})_i| \geqslant 1$. Since $\delta := \min_{\ve{\alpha} \in \mathbb{R}^N \backslash \{0\}}$ $\min \left\{ P[\ve{\alpha}^T \ve{\varepsilon} > 0], P[\ve{\alpha}^T \ve{\varepsilon} < 0] \right\} > 0$, with probability at least $\delta$ we have that
\begin{align*}
\sgn \left[ (\ve{A}^T \ve{\varepsilon})_{K^C} - \ve{C}_{K^C,K} \ve{C}_{K,K}^{-1}(\ve{A}^T\ve{\varepsilon})_K \right]_i = - \sgn(\ve{C}_{K^C,K} \ve{C}_{K,K}^{-1} \ve{s})_i.
\end{align*}
Under this event, condition~\eqref{eq:40} does not hold, which contradicts the assumption that $\sgn(\ve{x}) = \sgn(\ve{x}_0)$ (since $|w_i| \leqslant 1$). This concludes the proof. \qed

% -----------------------------------------------------
\subsection{Proof of Lemma~\ref{lem:inconsistency2}}
We will establish this result by embedding the random walk $s_t$ into a Brownian motion process. To this end, first notice that this probability is independent of the values of $\mu$ and $\sigma^2$, so we may suppose without loss of generality that $\mu = 0$ and $\sigma^2 = 1/N$. Now, let $W$ be a standard Brownian motion process~\citep{Karlin-66}. Since $\{s_1, \dots, s_N\}$ has the same joint distribution as $\{W(1/N), \ldots, W(1)\}$, it follows that
\begin{align*}
Q %
&= P\left\{W\left( \frac{t}{N} \right) - \frac{t}{N} W(1) \geqslant 0 \;\forall t \in \{1, \dots, \lfloor \varepsilon N \rfloor\} \cup \{N - \lfloor \varepsilon N \rfloor + 1, \ldots, N\}\right\} \\
&\geqslant P\{W(t) - t W(1) \geqslant 0\;\forall t \in [1/N,\lfloor \varepsilon N \rfloor/N] \cup [1 - \lfloor \varepsilon N \rfloor/N + 1/N, 1]\}.
\end{align*}
The process $B(t) := W(t) - tW(1)$, $t \in [0,1]$, is a Brownian bridge \citep{Karlin-Taylor-81}[eq.~(9.31)], and an alternative representation for such a process is $B(t) = (1 - t) W(t / (1 - t))$ \citep{Karlin-Taylor-81}[eq.~(9.29)]. This gives
\begin{align*}
Q %
&\geqslant P\left\{ (1 - t)W\left( \frac{t}{1 - t} \right) \geqslant 0 \; \forall t \in \left[\frac{1}{N}, \frac{\lfloor \varepsilon N \rfloor}{N}\right] \cup \left[1 - \frac{\lfloor \varepsilon N \rfloor}{N} + \frac{1}{N}, 1 \right] \right\} \\
&= P\left\{ W(t) \geqslant 0 \; \forall t \in \left[ \frac{1}{N - 1},\frac{\varepsilon'}{1 - \varepsilon'} \right] \cup \left[ \frac{1 - \varepsilon'}{\varepsilon'}, N - 1 \right] \right\},
\end{align*}
where $\lfloor \varepsilon N \rfloor / N =: \varepsilon'$. To compute this last probability, we appeal to \citep{Karlin-66}[pp.~278], which gives
\begin{align*}
&P\left\{ W(t) \geqslant 0 \; \forall t \in [a,b] \cup \left[ \frac{1}{b}, \frac{1}{a} \right] \right\} \\
&= \int\limits_0^\infty  dx_a \int\limits_0^\infty  dx_b \int\limits_0^\infty  dx_{1/b} \int\limits_0^\infty dx_{1/a} P\{ W(a) \in [x_a, x_a + dx_a)\}  \\
& \quad \times P\{ W(b) \in [x_b, x_b + dx_b),\;W(t) \geqslant 0 \; \forall t \in [a,b] | W(a) \in [x_a, x_a + dx_a)\} \\
&\quad \times P\{ W(1/b) \in [x_{1/b}, x_{1/b} + dx_{1/b}) | W(b) \in [x_b, x_b + dx_b)\}  \\
&\quad \times P\left\{ \left. W(1/a) \in [x_{1/a}, x_{1/a} + dx_{1/a}),\;W(t) \geqslant 0 \; \forall t \in \left[ \frac{1}{b}, \frac{1}{a} \right] \right| \right. \\
&\qquad \qquad \qquad \qquad \qquad \qquad \qquad \qquad \qquad W(1/b) \in [x_{1/b}, x_{1/b} + dx_{1/b}) \bigg\} \\
&= \int\limits_0^\infty  dx_a \int\limits_0^\infty  dx_b \int\limits_0^\infty  dx_{1/b} \int\limits_0^\infty  dx_{1/a} \frac{1}{\sqrt{2 \pi a}} \frac{1}{\sqrt{2 \pi (b - a)}} \frac{1}{\sqrt{2 \pi (b^{ - 1} - b)}} \frac{1}{\sqrt{2 \pi (a^{-1} - b^{-1})}} \\
&\quad \times \exp \left( -\frac{x_a^2}{2 a} \right) \left[ \exp \left( - \frac{(x_a - x_b)^2}{2(b - a)} \right) - \exp \left( -\frac{(x_a + x_b)^2}{2 (b - a)} \right) \right] \\
&\quad \times \exp \left( -\frac{(x_{1/b} - x_b)^2}{2(b^{-1} - b)} \right)\left[ \exp \left( -\frac{(x_{1/b} - x_{1/a})^2}{2 (a^{-1} - b^{-1})} \right) - \exp \left( -\frac{(x_{1/b} + x_{1/a})^2}{2(a^{-1} - b^{-1})} \right) \right] \\
&= P_1 - P_2 - P_3 + P_4,
\end{align*}
where
\begin{multline*}
P_1 := \frac{1}{(2 \pi)^2 \sqrt{a (b - a)(b^{-1} - b)(a^{-1} - b^{-1})}} \int\limits_0^\infty  dx_a \int\limits_0^\infty dx_b \int\limits_0^\infty dx_{1/b} \int\limits_0^\infty dx_{1/a} \\
\exp \left( -\frac{x_a^2}{2a} - \frac{(x_{1/b} - x_b)^2}{2(b^{-1} - b)} - \frac{(x_a - x_b)^2}{2(b - a)} - \frac{(x_{1/b} - x_{1/a})^2}{2(a^{-1} - b^{-1})} \right),
\end{multline*}
\begin{multline*}
P_2 := \frac{1}{(2 \pi)^2 \sqrt{a (b - a)(b^{-1} - b)(a^{-1} - b^{-1})}} \int\limits_0^\infty  dx_a \int\limits_0^\infty dx_b \int\limits_0^\infty dx_{1/b} \int\limits_0^\infty dx_{1/a} \\
\exp \left( -\frac{x_a^2}{2a} - \frac{(x_{1/b} - x_b)^2}{2(b^{-1} - b)} - \frac{(x_a + x_b)^2}{2(b - a)} - \frac{(x_{1/b} - x_{1/a})^2}{2(a^{-1} - b^{-1})} \right),
\end{multline*}
\begin{multline*}
P_3 := \frac{1}{(2 \pi)^2 \sqrt{a (b - a)(b^{-1} - b)(a^{-1} - b^{-1})}} \int\limits_0^\infty  dx_a \int\limits_0^\infty dx_b \int\limits_0^\infty dx_{1/b} \int\limits_0^\infty dx_{1/a} \\
\exp \left( -\frac{x_a^2}{2a} - \frac{(x_{1/b} - x_b)^2}{2(b^{-1} - b)} - \frac{(x_a - x_b)^2}{2(b - a)} - \frac{(x_{1/b} + x_{1/a})^2}{2(a^{-1} - b^{-1})} \right),
\end{multline*}
\begin{multline*}
P_4 := \frac{1}{(2 \pi)^2 \sqrt{a (b - a)(b^{-1} - b)(a^{-1} - b^{-1})}} \int\limits_0^\infty  dx_a \int\limits_0^\infty dx_b \int\limits_0^\infty dx_{1/b} \int\limits_0^\infty dx_{1/a} \\
\exp \left( -\frac{x_a^2}{2a} - \frac{(x_{1/b} - x_b)^2}{2(b^{-1} - b)} - \frac{(x_a + x_b)^2}{2(b - a)} - \frac{(x_{1/b} + x_{1/a})^2}{2(a^{-1} - b^{-1})} \right).
\end{multline*}
Now, notice that, using Lemma~\ref{lem:inconsistence_A} of Appendix~\ref{app:B},
 \begin{align*}
&\frac{1}{\sqrt {2 \pi (a^{-1} - b^{-1})}} \int\limits_0^\infty  dx_{1/a} \exp \left( - \frac{(x_{1/b} \mp x_{1/a})^2}{2(a^{-1} - b^{-1})} \right) \\
&\qquad \qquad \qquad \qquad = \frac{1}{\sqrt{2\pi (a^{-1} - b^{-1}})} \int\limits_{\mp \frac{x_{1/b}}{\sqrt{a^{-1} - b^{-1}}}}^\infty dx_{1/a} \exp \left( - \frac{x_{1/a}^2}{2} \right) \\
&\qquad \qquad \qquad \qquad = \frac{1}{2} \pm \frac{x_{1/b}}{\sqrt{2 \pi (a^{-1} - b^{-1})}} + O\left( \frac{x_{1/b}^3}{(a^{-1} - b^{-1})^{3/2}} \right)
\end{align*}
and
\begin{align*}
&\frac{1}{2\pi \sqrt{a (a - b)}} \int\limits_0^\infty  dx_a \exp \left( -\frac{x_a^2}{2a} - \frac{(x_a \mp x_b)^2}{2(b - a)} \right) \\
&\qquad = \frac{1}{2 \pi \sqrt{a (b - a)}} \int\limits_0^\infty dx_a \exp \left( - \frac{x_a^2}{2a} - \frac{x_a^2 \mp 2 x_a x_b + x_b^2}{2 (b - a)} \right) \\
&\qquad = \frac{1}{2 \pi \sqrt{a (b - a)}} \exp \left(  - \frac{x_b^2}{2b} \right) \int\limits_0^\infty dx_a \exp \left( -\frac{b}{2 a (b - a)}{\left( x_a \mp \frac{a}{b} x_b \right)^2} \right) \\
&\qquad = \frac{1}{\sqrt{2 \pi b}} \exp \left( - \frac{x_b^2}{2 b} \right) \frac{1}{\sqrt{2 \pi}} \int\limits_{\mp \sqrt{\frac{a}{b (b - a)}} x_b}^\infty dx_a \exp \left( - \frac{x_a^2}{2} \right) \\
&\qquad = \frac{1}{\sqrt {2 \pi b}} \exp \left( -\frac{x_b^2}{2b} \right) \left[ \frac{1}{2} \pm \sqrt{\frac{a}{2 \pi b (b - a)}} x_b + O\left( \frac{a^{3/2}}{[b (b - a)]^{3/2}} x_b^3 \right) \right].
\end{align*}
Therefore,
\begin{align*}
P_{1,2,3,4} &= \frac{1}{2 \pi \sqrt{b(b^{-1} - b)}} \int\limits_0^\infty dx_b \int\limits_0^\infty dx_{1/b} \exp \left( -\frac{x_b^2}{2b} - \frac{(x_{1/b} - x_b)^2}{2 (b^{-1} - b)} \right) \\
&\qquad \qquad \times \left[ \frac{1}{2} \pm \frac{x_{1/b}}{\sqrt{2 \pi (a^{-1} - b^{-1})}} + O\left( \frac{x_{1/b}^3}{(a^{-1} - b^{-1})^{3/2}} \right) \right] \\
&\qquad \qquad \times \left[ \frac{1}{2} \pm \sqrt{\frac{a}{2 \pi b (b - a)}} x_b + O\left( \frac{a^{3/2}}{[b (b - a)]^{3/2}} x_b^3 \right) \right].
\end{align*}
Performing the integration with respect to $x_{1/b}$ gives
\begin{align*}
&\frac{1}{\sqrt{2 \pi (b^{-1} - b)}} \int\limits_0^\infty dx_{1/b} \exp \left( -\frac{(x_{1/b} - x_b)^2}{2(b^{-1} - b)} \right) \\
&\qquad \qquad \qquad \qquad \qquad \times \left[ \frac{1}{2} \pm \frac{x_{1/b}}{\sqrt{2 \pi (a^{-1} - b^{-1})}} + O\left( \frac{x_{1/b}^3}{(a^{-1} - b^{-1})^{3/2}} \right) \right] \\
&= \frac{1}{\sqrt {8 \pi}} \int\limits_{-\frac{x_b}{\sqrt{b^{-1} - b}}}^\infty dx_{1/b} \exp \left( -\frac{x_{1/b}^2}{2} \right) \\
&\quad \pm \frac{1}{2\pi \sqrt{a^{-1} - b^{-1}}} \int\limits_{-\frac{x_b}{\sqrt {b^{-1} - b}}}^\infty dx_{1/b} (x_{1/b} \sqrt{b^{-1} - b} + x_b) \exp \left( -\frac{x_{1/b}^2}{2} \right) + O(a^{3/2}) \\
&= \frac{1}{2} \Phi \left( \frac{x_b}{\sqrt{b^{-1} - b}} \right) + O(a^{3/2}) \\
&\quad \pm \left[ \frac{1}{2 \pi}\sqrt{\frac{b^{- } - b}{a^{-1} - b^{-1}}} \exp \left( -\frac{x_b^2}{2(b^{-1} - b)} \right) + \frac{x_b}{\sqrt{2 \pi (a^{-1} - b^{-1})}} \Phi \left( \frac{x_b}{\sqrt{b^{-1} - b}} \right) \right].
\end{align*}
Hence,
\begin{align*}
&P_{1,2,3,4} \\
&= \frac{1}{\sqrt{2 \pi b}} \int\limits_0^\infty dx_b \exp \left( -\frac{x_b^2}{2b} \right) \left\{ \frac{1}{2} \Phi \left( \frac{x_b}{\sqrt{b^{-1} - b}} \right) + O(a^{3/2}) \right. \\
&\quad \left. \pm \left[ \frac{1}{2 \pi} \sqrt{\frac{b^{-1} - b}{a^{-1} - b^{-1}}} \exp \left( -\frac{x_b^2}{2(b^{-1} - b)} \right) + \frac{x_b}{\sqrt{2 \pi (a^{-1} - b^{-1})}} \Phi \left( \frac{x_b}{\sqrt{b^{-1} - b}} \right) \right] \right\} \\
&\qquad \times \left\{ \frac{1}{2} \pm \sqrt{\frac{a}{2 \pi b (b - a)}} x_b + O\left( \frac{a^{3/2}}{[b (b - a)]^{3/2}} x_b^3 \right) \right\} \\
&= \frac{1}{4} \frac{1}{\sqrt{2\pi b}} \int\limits_0^\infty dx_b \exp \left( - \frac{x_b^2}{2b} \right) \Phi \left( \frac{x_b}{\sqrt{b^{-1} - b}} \right) \\
&\pm \frac{1}{\sqrt{2 \pi b}} \int\limits_0^\infty dx_b \exp \left( -\frac{x_b^2}{2b} \right) \frac{1}{2} \Phi \left( \frac{x_b}{\sqrt{b^{-1} - b}} \right) \sqrt{\frac{a}{2 \pi b (b - a)}} x_b \\
&\pm \frac{1}{\sqrt{8 \pi b}} \int\limits_0^\infty dx_b \exp \left( -\frac{x_b^2}{2b} \right) \left[ \frac{1}{2 \pi} \sqrt{\frac{b^{-1} - b}{a^{-1} - b^{-1}}} \exp \left( -\frac{x_b^2}{2 (b^{-1} - b)} \right) \right. \\
&\qquad \qquad \qquad \qquad \qquad \qquad \qquad \qquad \left. + \frac{x_b}{\sqrt{2 \pi (a^{-1} - b^{-1})}} \Phi \left( \frac{x_b}{\sqrt{b^{-1} - b}} \right) \right] \\
&\pm \frac{1}{\sqrt{2 \pi b}} \int\limits_0^\infty dx_b \exp \left( -\frac{x_b^2}{2b} \right) \sqrt{\frac{a}{2 \pi b (b - a)}} x_b \\
&\quad \times \left[ \frac{1}{2 \pi} \sqrt{\frac{b^{-1} - b}{a^{-1} - b^{-1}}} \exp \left( -\frac{x_b^2}{2 (b^{-1} - b)} \right) + + \frac{x_b}{\sqrt{2 \pi (a^{-1} - b^{-1})}} \Phi \left( \frac{x_b}{\sqrt{b^{-1} - b}} \right) \right] \\
&+ O(a^{3/2}).
\end{align*}
The first three terms of this last expression for $P_{1,2,3,4}$ will cancel out when summed up in $P\{ W(t) \geqslant 0 \; \forall t \in [a,b] \cup [b^{-1}, a^{-1}]\}$, while the fourth term will appear four times. Therefore,
\begin{align*}
&P\left\{ W(t) \geqslant 0\; \forall t \in [a,b] \cup \left[ \frac{1}{b},\frac{1}{a} \right] \right\} \\
&= O(a^{3/2}) + \frac{4}{\sqrt{2 \pi b}} \int\limits_0^\infty dx_b \exp \left( -\frac{x_b^2}{2b} \right) \sqrt{\frac{a}{2 \pi b (b - a)}} x_b \\
& \quad \times \left[ \frac{1}{2 \pi} \sqrt{\frac{b^{-1} - b}{a^{-1} - b^{-1}}} \exp \left( -\frac{x_b^2}{2 (b^{-1} - b)} \right) + \frac{x_b}{\sqrt{2 \pi (a^{-1} - b^{-1})}} \Phi \left( \frac{x_b}{\sqrt{b^{-1} - b}} \right) \right] \\
&= \frac{2}{\pi b} \sqrt{\frac{a}{b - a}} \int\limits_0^\infty dx_b \left[ \frac{1}{2 \pi} \sqrt{\frac{b^{-1} - b}{a^{-1} - b^{-1}}} x_b \exp \left( -\frac{x_b^2}{2 b (1 - b^2)} \right) + \right. \\
&\qquad \qquad \qquad \qquad \quad + \left. \frac{x_b^2}{\sqrt{2 \pi (a^{-1} - b^{-1})}} \exp \left( -\frac{x_b^2}{2 b} \right) \Phi \left( \frac{x_b}{\sqrt{b^{-1} - b}} \right) \right] + O(a^{3/2}) \\
&= \frac{a}{\pi^2 b} \sqrt{\frac{1}{b - a}} \sqrt{\frac{b^{-1} - b}{1 - a b^{-1}}} \int\limits_0^\infty dx_b x_b \exp \left( -\frac{x_b^2}{2 b (1 - b^2)} \right) + \\
&\qquad + \frac{\sqrt{2} a}{\pi^{3/2} \sqrt{b} (b - a)} \int\limits_0^\infty dx_b x_b^2\exp \left( - \frac{x_b^2}{2 b} \right) \Phi \left( \frac{x_b}{\sqrt{b^{-1} - b}} \right) + O(a^{3/2}) \\
&= \frac{a}{\pi^2 b} \frac{(1 - b^2)^{3/2}}{1 - a b^{-1}} + \frac{\sqrt{2} a}{\pi^{3/2} \sqrt{b} (b - a)} \int\limits_0^\infty dx_b x_b^2 \exp \left( -\frac{x_b^2}{2 b} \right) \Phi \left( \frac{x_b}{\sqrt{b^{-1} - b}} \right) + O(a^{3/2}).
\end{align*}
Then, using Lemma~\ref{lem:inconsistence_B} of Appendix~\ref{app:B}, we obtain
%
%\begin{multline*}
%\int\limits_0^\infty dx_b x_b^2 \exp \left( -\frac{x_b^2}{2 b} \right) \Phi \left( \frac{x_b}{\sqrt{b^{-1} - b}} \right) \\
%= \frac{b^{3/2}}{\sqrt{2 \pi}} \left[ \pi  - \arctan\left(\frac{\sqrt{1 - b^2}}{b}\right) + b \sqrt{1 - b^2} \right],
%\end{multline*}
%%
%so
%
\begin{align*}
&P\left\{ W(t) \geqslant 0 \; \forall t \in [a,b] \cup \left[ \frac{1}{b},\frac{1}{a} \right] \right\} \\
&= \frac{a}{\pi^2 b} \frac{(1 - b^2)^{3/2}}{1 - a b^{-1}} + \frac{a}{\pi^2 (1 - a b^{-1})} \left[ \pi  - \arctan \left( \frac{\sqrt{1 - b^2}}{b} \right) + b \sqrt{1 - b^2} \right] + O(a^{3/2}) \\
&= \frac{a}{b} \left[ \frac{1}{\pi^2} + o(a) + O(b) \right].
\end{align*}
This result gives us the desired lower bound for $Q$, by letting $a = 1/N$ and $b = \varepsilon'$, since $|\varepsilon' - \varepsilon| < 1/N$. \qed

% ==============================
\section{Technical definitions and lemmas} \label{app:B}

\begin{defi}[Sub-exponential random variables, \citep{Vershynin-12}]
A random variable $x$ is said to be \emph{sub-exponential} if any of the following three conditions is met:
\begin{itemize}
\item[(a)] $P\{|x| > t\} \leqslant \exp(1 - t/K_1)$, for all $t \geqslant 0$,

\item[(b)] $(E\{|x|^p\})^{1/p} \leqslant K_2 p$, for all $p \geqslant 1$,

\item[(c)] $E\{\exp(x/K_3)\} \leqslant e$,
\end{itemize}
where $K_1, K_2, K_3 > 0$ are arbitrary constants. In this case, the \emph{sub-exponential norm} of $x$ is defined as $\|x\|_{\psi_1} := \sup _{p \geqslant 1} p^{-1} (E\{|x|^p\})^{1/p}$.
\end{defi}

\begin{remm}
The term ``sub-exponential distribution'' has unfortunately another standard and almost opposite interpretation in probability (in particular, in queueing theory) than the one given here: a sub-exponential distribution is also a class of heavy-tailed distributions (\ie those distributions $F$ whose moment generating function $M_F(t) := \int_{-\infty}^\infty e^{t x}dF(x)$ is infinite for every $t > 0$). Our definition comes from the theory of random matrices \citep{Vershynin-12}.	
\end{remm}

\begin{remm}
Notice that the class of sub-exponential random variables is reasonably large, and it includes for instance all Gaussian, Bernoulli, exponential, chi-square and bounded random variables.
\end{remm}

\begin{lem}[Perturbation of the cumulative normal distribution function] \label{lem:inconsistence_A}
For $x \in \mathbb{R}$,
\begin{align*}
\frac{1}{\sqrt{2 \pi}} \int\limits_x^\infty \exp \left( -\frac{t^2}{2} \right) dt = \frac{1}{2} - \frac{1}{\sqrt{2 \pi}} x + \frac{1}{\sqrt{72 \pi}} x^3 + O(x^5).
\end{align*}
\end{lem}

\begin{proof}
\begin{align*}
\frac{1}{\sqrt{2 \pi}} \int\limits_x^\infty \exp \left( -\frac{t^2}{2} \right) dt %
&= \frac{1}{\sqrt{2 \pi}} \int\limits_0^\infty \exp \left( -\frac{t^2}{2} \right) dt - \frac{1}{\sqrt{2 \pi}} \int\limits_0^x \exp \left( -\frac{t^2}{2} \right) dt \\
&= \frac{1}{2} - \frac{1}{\sqrt{2 \pi}} \int\limits_0^x \left[ 1 - \frac{t^2}{2} + O(t^4) \right] dt \\
&= \frac{1}{2} - \frac{1}{\sqrt{2 \pi}} \left[ x - \frac{x^3}{6} + O(x^5) \right] \\
&= \frac{1}{2} - \frac{1}{\sqrt{2 \pi}} x + \frac{1}{\sqrt{72 \pi}}{x^3} + O(x^5).
\end{align*}
\end{proof}

\begin{lem}[Integral] \label{lem:inconsistence_B}
Let $b \in (0,1)$, and $\Phi$ be the cumulative standard normal distribution function. Then,
\begin{align*}
\int\limits_0^\infty dx x^2 \exp \left( -\frac{x^2}{2 b} \right) \Phi \left( \frac{x}{\sqrt{b^{-1} - b}} \right) = \frac{b^{3/2}}{\sqrt{2 \pi}} \left[ \pi  - \arctan\left( \frac{\sqrt{1 - b^2}}{b} \right) + b \sqrt{1 - b^2} \right].
\end{align*}
\end{lem}

\begin{proof}
\begin{align*}
&\int\limits_0^\infty dx x^2 \exp \left( -\frac{x^2}{2 b} \right) \Phi \left( \frac{x}{\sqrt{b^{-1} - b}} \right) \\
&\quad = \frac{1}{\sqrt{2 \pi}} \int\limits_0^\infty dx \int\limits_{-\infty}^{\frac{x}{\sqrt{b^{-1} - b}}} dy x^2 \exp \left( -\frac{x^2}{2 b} - \frac{y^2}{2} \right) \\
&\quad = \frac{b^{3/2}}{\sqrt{2 \pi}} \int\limits_0^\infty dx \int\limits_{-\infty}^{\frac{b x}{\sqrt{1 - b^2}}} dy x^2 \exp \left( -\frac{x^2 + y^2}{2} \right) \\
&\quad = \frac{b^{3/2}}{\sqrt{2 \pi}} \int\limits_0^\infty dr \int\limits_{-\pi / 2}^{\arctan(b / \sqrt{1 - b^2})} d\theta \,r r^2 \cos^2(\theta) \exp \left( -\frac{r^2}{2} \right) \\
&\quad = \frac{b^{3/2}}{\sqrt{2 \pi}} \left[ \int\limits_{-\pi / 2}^{\arctan(b/\sqrt{1 - b^2})} d\theta \frac{1 + \cos(2 \theta)}{2} \right] \cdot 2 \int\limits_0^\infty du \,u \exp(-u) \\
&\quad = \frac{b^{3/2}}{\sqrt{2 \pi}} \left[ \theta + \frac{1}{2} \sin(2 \theta) \right]_{-\pi / 2}^{\arctan(b / \sqrt{1 - b^2})} \left[ \left. - u e^{-u} \right|_{u = 0}^\infty + \int\limits_0^\infty du \exp(-u) \right] \\
%&\quad = \frac{b^{3/2}}{\sqrt{2 \pi}} \left[ \arctan(b / \sqrt{1 - b^2}) + \frac{\pi}{2} + \frac{1}{2} \sin(2 \arctan(b / \sqrt{1 - b^2})) \right] \cdot \\
%&\qquad \qquad \cdot \left. (-e^{-u}) \right|_{u = 0}^\infty \\
&\quad = \frac{b^{3/2}}{\sqrt{2 \pi}} \left[ \arctan(b / \sqrt{1 - b^2}) + \frac{\pi}{2} \right. \\
&\qquad \qquad \qquad \left. + \sin(\arctan(b / \sqrt{1 - b^2})) \cos(\arctan(b / \sqrt{1 - b^2})) \right] \\
&\quad = \frac{b^{3/2}}{\sqrt{2 \pi}} \left[ \arctan(b / \sqrt{1 - b^2}) + \frac{\pi}{2} + b \sqrt{1 - b^2} \right] \\
&\quad = \frac{b^{3/2}}{\sqrt{2 \pi}} \left[ \pi - \arctan(\sqrt{1 - b^2} / b) + b\sqrt{1 - b^2} \right].
\end{align*}
\end{proof}

\begin{lem}[Bound for weighted sub-exponential sums] \label{lem:B1}
Let $\ve{\varepsilon} \in \mathbb{R}^N$ be a vector of independent zero mean sub-exponential random variables, with $\kappa := \max_{1 \leqslant i \leqslant N} \|\varepsilon_i\|_{\psi_1}$, and let $\ve{\alpha} \in \mathbb{R}^N$ be a deterministic vector. Then, for every $x \geqslant 0$,
\begin{align*}
P\left\{ |\ve{\alpha}^T \ve{\varepsilon}| \geqslant x \right\} \leqslant 2 \exp \left[ -c \min \left( \frac{x^2}{\kappa^2 \| \ve{\alpha} \|_2^2}, \frac{x}{\kappa \| \ve{\alpha} \|_\infty} \right) \right],
\end{align*}
where $c > 0$ is an absolute constant.
\end{lem}

\begin{proof}
See \citep{Vershynin-12}[Proposition~5.16].
\end{proof}

\begin{lem}[Min-max probability bound for sub-exponential sums] \label{lem:B2}
Let $\ve{\varepsilon} \in \mathbb{R}^N$ be a vector of independent zero mean sub-exponential continuous random variables, with $\kappa := \max_{1 \leqslant i \leqslant N} \|\varepsilon_i\|_{\psi_1}$. Then, for every $\lambda > 0$ and $0 < \delta < 1$:
\begin{multline*}
P\left\{ \min\limits_{1 \leqslant t_1 \leqslant \delta N} \max\limits_{t_1 \leqslant t_2 \leqslant N} \frac{1}{t_2 - t_1 + 1} \sum\limits_{t = t_1}^{t_2} \varepsilon_t  > \lambda \right\} \\
\leqslant 2 (1 - \delta) N \exp \left[  -(1 - \delta ) N c \min \left( \frac{\lambda^2}{\kappa^2}, \frac{\lambda}{\kappa} \right) \right],
\end{multline*}
where $c > 0$ is an absolute constant.
\end{lem}

\begin{proof}
Let
\begin{align*}
Q := P\left\{ \min\limits_{1 \leqslant t_1 \leqslant \delta N} \max\limits_{t_1 \leqslant t_2 \leqslant N} \frac{1}{t_2 - t_1 + 1} \sum\limits_{t = t_1}^{t_2} \varepsilon_t < \lambda \right\}.
\end{align*}
Our goal is to obtain an upper bound for $1 - Q$. To this end, notice that
\begin{align} \label{eq:41}
&Q \\
&= P\left\{ \exists 1 \leqslant t_1 \leqslant \delta N \text{ s.t. } \max\limits_{t_1 \leqslant t_2 \leqslant N} \frac{1}{t_2 - t_1 + 1} \sum\limits_{t = t_1}^{t_2} \varepsilon_t < \lambda \right\} \nonumber \\
&= P\left\{ \exists 0 \leqslant t_1 \leqslant \delta N - 1 \text{ s.t. } \forall t_1 < t_2 \leqslant N, \sum\limits_{t = t_1 + 1}^{t_2} \varepsilon_t  < \lambda (t_2 - t_1) \right\} \nonumber \\
&= P\left\{ \exists 0 \leqslant t_1 \leqslant \delta N - 1 \text{ s.t. } \forall t_1 < t_2 \leqslant N, \sum\limits_{t = 1}^{t_2} \varepsilon_t < \lambda (t_2 - t_1) + \sum\limits_{t = 1}^{t_1} \varepsilon_t \right\} \nonumber \\
&= P\left\{ \exists (1 - \delta) N + 1 \leqslant \tilde{t}_1 \leqslant N \text{ s.t. } \forall 0 \leqslant \tilde{t}_2 < \tilde{t}_1, \sum\limits_{t = 1}^{\tilde{t}_1} (\tilde{\varepsilon}_t - \lambda) < \sum\limits_{t = 1}^{\tilde{t}_2} (\tilde{\varepsilon}_t - \lambda) \right\} \nonumber \\
&= P\left\{ \exists (1 - \delta) N + 1 \leqslant \tilde{t}_1 \leqslant N \text{ s.t. } \sum\limits_{t = 1}^{\tilde{t}_1} (\tilde{\varepsilon}_t - \lambda) < \min\limits_{0 \leqslant \tilde{t}_2 \leqslant \tilde{t}_1 - 1} \sum\limits_{t = 1}^{\tilde{t}_2} (\tilde{\varepsilon}_t - \lambda) \right\}, \nonumber
\end{align}
%(41)
where $\tilde{\ve{\varepsilon}} \in \mathbb{R}^N$ is a random vector given by $\tilde{\varepsilon}_t := \varepsilon_{N + 1 - t}$. Furthermore, in the last line of \eqref{eq:41} we can restrict the range of $\tilde{t}_2$ to $\{0, \ldots, (1 - \delta) N\}$, since if for some $(1 - \delta) N + 1 \leqslant \tilde{t}_1 \leqslant N$ the minimizing $\tilde{t}_2 = \tilde{t}_2^*$ is larger than $(1 - \delta) N$, then the inequality obviously holds by taking   $\tilde{t}_2 = \tilde{t}_2^*$ and $\tilde{t}_2 = (1 - \delta) N$. Conversely, if
\begin{align*}
\sum\limits_{t = 1}^{\tilde{t}_1} (\tilde{\varepsilon}_t - \lambda) < \min\limits_{0 \leqslant \tilde{t}_2 \leqslant (1 - \delta) N} \sum\limits_{t = 1}^{\tilde{t}_2} (\tilde{\varepsilon}_t - \lambda),
\end{align*}
for some $\tilde{t}_1 = \tilde{t}_1^* \in \{(1 - \delta) N + 1, \ldots, N\}$, then it also holds that
\begin{align*}
\min\limits_{(1 - \delta) N + 1 \leqslant \tilde{t}_1 \leqslant N} \displaystyle \sum\limits_{t = 1}^{\tilde{t}_1} (\tilde{\varepsilon}_t - \lambda) <  \min\limits_{0 \leqslant \tilde{t}_2 \leqslant \tilde{t}_1 - 1} \displaystyle \sum\limits_{t = 1}^{\tilde{t}_2} (\tilde{\varepsilon}_t - \lambda).
\end{align*}
Therefore,
\begin{align*}
Q &= P\left\{ \exists (1 - \delta) N + 1 \leqslant \tilde{t}_1 \leqslant N \text{ s.t. } \sum\limits_{t = 1}^{\tilde{t}_1} (\tilde {\varepsilon}_t - \lambda) < \min\limits_{0 \leqslant \tilde{t}_2 \leqslant (1 - \delta) N} \sum\limits_{t = 1}^{\tilde{t}_2} (\tilde {\varepsilon}_t - \lambda) \right\} \\
&= P\left\{ \min\limits_{(1 - \delta) N + 1 \leqslant \tilde{t}_1 \leqslant N} \sum\limits_{t = 1}^{\tilde{t}_1} (\tilde{\varepsilon}_t - \lambda)  < \min\limits_{0 \leqslant \tilde{t}_2 \leqslant (1 - \delta) N} \sum\limits_{t = 1}^{\tilde{t}_2} (\tilde{\varepsilon}_t - \lambda) \right\}.
\end{align*}
This last expression shows that $Q$ is the probability that the minimum of the random walk $\left\{ \sum\nolimits_{k = 1}^t (\tilde{\varepsilon}_k - \lambda) \right\}_t$ lies in $\{(1 - \delta) N + 1, \ldots, N\}$ (since the $\varepsilon_t$'s have continuous distributions, the probability that such random walk attains its minimum at more than one time instant is zero). This quantity can be computed in principle by using techniques from fluctuation theory (see, \eg \citep{Feller-62}). In particular, if we denote $\sum\nolimits_{k = a}^b (\tilde{\varepsilon}_k - \lambda)$ by $S(a, b)$, then, inspired by \citep{Spitzer-56}[equation~(5.3)], we have that
\begin{align} \label{eq:42}
1 - Q &= P\left\{ \min\limits_{(1 - \delta) N + 1 \leqslant \tilde{t}_1 \leqslant N} S(1, \tilde{t}_1) > \min\limits_{0 \leqslant \tilde{t}_2 \leqslant (1 - \delta) N} S(1, \tilde{t}_2) \right\} \\
&= \sum\limits_{t = 1}^{(1 - \delta) N} P\left\{ S(1,t) = \min\limits_{0 \leqslant \tilde{t} \leqslant N} S(1, \tilde{t}) \right\} \nonumber \\
&= \sum\limits_{t = 1}^{(1 - \delta) N} P\left\{ S(1,t) < \min\limits_{0 \leqslant \tilde{t} \leqslant t - 1} S(1, \tilde{t}) \text{ and } S(1,t) < \min\limits_{t + 1 \leqslant \tilde{t} \leqslant N} S(1, \tilde{t}) \right\} \nonumber \\
&= \sum\limits_{t = 1}^{(1 - \delta) N} P\left\{ \max\limits_{0 \leqslant \tilde{t} \leqslant t - 1} S(\tilde{t} + 1, t) < 0 \text{ and } \min\limits_{t + 1 \leqslant \tilde{t} \leqslant N} S(t + 1, \tilde{t}) > 0 \right\} \nonumber \\
&\leqslant \sum\limits_{t = 1}^{(1 - \delta) N} P\left\{ \min\limits_{t + 1 \leqslant \tilde{t} \leqslant N} S(t + 1, \tilde{t}) > 0 \right\}, \nonumber
\end{align}
%(42)
where in the fourth equality we have used the additive property of $S$, namely, that $S(a,b) + S(b + 1,c) = S(a,c)$ for all $a \leqslant b \leqslant c$. Notice that, for all $t \in \mathbb{N}$,
\begin{align*}
P\left\{ \min\limits_{t + 1 \leqslant \tilde{t} \leqslant N} S(t + 1, \tilde{t}) > 0 \right\} %
&= P\left\{ \min\limits_{t + 1 \leqslant \tilde{t} \leqslant N} \sum\limits_{k = t + 1}^{\tilde{t}} (\tilde{\varepsilon}_k - \lambda)  > 0 \right\} \\
&= P\left\{ \sum\limits_{k = t + 1}^{\tilde{t}} (\tilde{\varepsilon}_k - \lambda) > 0, \; \forall t + 1 \leqslant \tilde{t} \leqslant N \right\} \\
&\leqslant P\left\{ \sum\limits_{k = t + 1}^N (\tilde{\varepsilon}_k - \lambda)  > 0 \right\} \\
&= P\left\{ \frac{1}{N - t} \sum\limits_{k = t + 1}^N \tilde{\varepsilon}_k > \lambda \right\} \\
&= P\left\{ \frac{1}{N - t} \sum\limits_{k = 1}^{N - t} \varepsilon_k > \lambda \right\}.
\end{align*}
Therefore, by Lemma~\ref{lem:B1} (taking $\alpha = [t^{-1} \; \cdots \; t^{-1}]^T \in \mathbb{R}^t$),
\begin{align*}
1 - Q %
&\leqslant \sum\limits_{t = \delta N}^{N - 1} P\left\{ \frac{1}{t} \sum\limits_{k = 1}^t \varepsilon_k > \lambda \right\} \\
&\leqslant \sum\limits_{t = \delta N}^{N - 1} 2 \exp \left[ -c t \min \left( \frac{\lambda^2}{\kappa^2}, \frac{\lambda}{\kappa} \right) \right] \\
&\leqslant 2 (1 - \delta) N \exp \left[ -(1 - \delta) N c \min \left( \frac{\lambda^2}{\kappa^2}, \frac{\lambda}{\kappa} \right) \right].
\end{align*}
This concludes the proof.
\end{proof}

\begin{cor}[Min-max probability bound for centered sub-exponential sums] \label{cor:B1}
Let $\ve{\varepsilon} \in \mathbb{R}^N$ be a vector of independent zero mean sub-exponential continuous random variables, with $\kappa := \max_{1 \leqslant i \leqslant N} \|\varepsilon_i\|_{\psi_1}$. Then, for every $\lambda > 0$ and $0 < \delta \leqslant 1/2$:
\begin{multline*}
P\left\{  \min\limits_{1 \leqslant t_1 \leqslant \delta N} \max\limits_{t_1 \leqslant
t_2 \leqslant N} \frac{1}{t_2 - t_1 + 1} \sum\limits_{t = t_1}^{t_2} (\varepsilon_t - \bar{\varepsilon}_{t_1}) > \lambda \right\} \\
\leqslant 2 [(1 - \delta) N + 1] \exp \left[ -\frac{c \delta N}{1 - \delta} \min \left( \frac{\lambda^2}{\kappa^2}, \frac{\lambda}{\kappa} \right) \right],
\end{multline*}
where $\bar{\varepsilon}_{t_1} := (N - t_1 + 1)^{-1} \sum\nolimits_{t = t_1}^N \varepsilon_t$, and $c > 0$ is an absolute constant.
\end{cor}

\begin{proof}
The proof of Lemma~\ref{lem:B2} carries over until \eqref{eq:42}, which changes to
\begin{align*}
1 - Q %
&= P\left\{ \min\limits_{(1 - \delta) N + 1 \leqslant \tilde{t}_1 \leqslant N} S(1, \tilde{t}_1) > \min\limits_{0 \leqslant \tilde{t}_2 \leqslant (1 - \delta) N} S(1, \tilde{t}_2) \right\} \\
&\leqslant \sum\limits_{t = 0}^{(1 - \delta) N} P\left\{ S(1,t) < \min\limits_{(1 - \delta) N + 1 \leqslant \tilde{t}\leqslant N} S(1,\tilde{t}) \right\} \\
&\leqslant \sum\limits_{t = 0}^{(1 - \delta) N} P\left\{ \min\limits_{(1 - \delta) N + 1 \leqslant \tilde{t} \leqslant N} S(t + 1,\tilde{t}) > 0 \right\} \\
&\leqslant \sum\limits_{t = 0}^{(1 - \delta) N} P\left\{ \min\limits_{(1 - \delta) N + 1 \leqslant \tilde{t} \leqslant N} \sum\limits_{k = t + 1}^{\tilde{t}} (\tilde{\varepsilon}_k - \tilde{\bar{\varepsilon}}_{\tilde{t}} - \lambda) > 0 \right\} \\
&= \sum\limits_{t = 0}^{(1 - \delta) N} P\left\{ \sum\limits_{k = t + 1}^{\tilde{t}} (\tilde{\varepsilon}_k - \tilde{\bar{\varepsilon}}_{\tilde{t}} - \lambda) > 0, \; \forall (1 - \delta) N + 1 \leqslant \tilde{t} \leqslant N \right\} \\
&\leqslant \sum\limits_{t = 0}^{(1 - \delta) N} P\left\{ \frac{1}{N - t} \sum\limits_{k = t + 1}^N (\tilde{\varepsilon}_k - \tilde{\bar{\varepsilon}}_N) > \lambda \right\} \\
&= \sum\limits_{t = 0}^{(1 - \delta) N} P\left\{ \frac{1}{N - t} \sum\limits_{k = 1}^{N - t} (\varepsilon_k - \tilde{\bar{\varepsilon}}_N) > \lambda \right\} \\
&= \sum\limits_{t = \delta N}^N P\left\{ \frac{1}{t} \sum\limits_{k = 1}^t (\varepsilon_k - \tilde{\bar{\varepsilon}}_N) > \lambda \right\},
\end{align*}
where $\tilde{\varepsilon}_t := \varepsilon_{N + 1 - t}$, $\tilde{\bar{\varepsilon}}_{\tilde{t}} := \tilde t^{-1} \sum\nolimits_{t = 1}^{\tilde{t}} \tilde{\varepsilon}_t$, and $S(a, b) := \sum\nolimits_{k = a}^b (\tilde{\varepsilon}_k - \lambda)$. The last expression can also be written as
\begin{align*}
1 - Q \leqslant \sum\limits_{t = \delta N}^N P\left\{ \ve{\alpha}_t^T \ve{\varepsilon} > \lambda \right\},
\end{align*}
with $\ve{\alpha}_t^T = [(t^{-1} - N^{-1}) \; \cdots \; (t^{-1} - N^{-1}) \;\ -N^{-1} \; \cdots \; -N^{-1}] \in \mathbb{R}^N$. Therefore, by Lemma~\ref{lem:B1} and the assumption that $\delta \leqslant 1 / 2$,
\begin{align*}
1 - Q %
&\leqslant \sum\limits_{t = \delta N}^N 2\exp \left[ -c \min \left( \frac{\lambda^2}{\kappa^2 \|\ve{\alpha}_t\|_2^2}, \frac{\lambda}{\kappa \|\ve{\alpha}_t\|_\infty} \right) \right] \\
&= 2 \sum\limits_{t = \delta N}^N \exp \left[ -c \min \left( \frac{\lambda^2}{\kappa^2 \left( \frac{1}{t} - \frac{1}{N} \right)},\frac{\lambda }{\kappa \max \left[ \left( \frac{1}{t} - \frac{1}{N} \right),\frac{1}{N} \right]} \right) \right] \\
&= 2 (N - \delta N + 1) \exp \left[ -c \min \left( \frac{\lambda^2}{\kappa^2 \left( \frac{1}{\delta N} - \frac{1}{N} \right)}, \frac{\lambda}{\kappa \max \left[ \left( \frac{1}{\delta N} - \frac{1}{N} \right), \frac{1}{N} \right]} \right) \right] \\
&= 2 [(1 - \delta) N + 1] \exp \left[ -\frac{c \delta N}{1 - \delta} \min \left( \frac{\lambda^2}{\kappa^2}, \frac{\lambda}{\kappa} \right) \right].
\end{align*}
This concludes the proof.
\end{proof}

% -----------------------------------------------------
\subsection{Proof of Lemma~\ref{lem:1}}
Problem~\eqref{eq:2} can be expressed as
\begin{align} \label{eq:30}
\begin{array}{cl}
\min\limits_{\begin{subarray}{l}
\eta_1, \ldots, \eta_N \\
\tau_2, \ldots, \tau_N
\end{subarray}} & \displaystyle -\frac{1}{2} \sum\limits_{t = 1}^N \ln(-\eta_t) - \sum\limits_{t = 1}^N \eta_t y_t^2 + \lambda \sum\limits_{t = 2}^N \tau_t \\
\text{s.t.} & \eta_t < 0, \quad t = 1, \ldots, N \\
               & \eta_t - \eta_{t - 1} \leqslant \tau_t, \quad t = 2, \ldots, N \\
               & \eta_{t - 1} - \eta_t \leqslant \tau_t, \quad t = 2, \ldots, N.
\end{array}
\end{align}
%(30)
The Lagrangian function of \eqref{eq:30} is
\begin{align*}
&L(\eta_1, \ldots, \eta_N, \tau_2, \ldots, \tau_N, \mu_2^1, \ldots,\mu_N^1, \mu _2^2, \ldots, \mu_N^2) \\
&= \sum\limits_{t = 1}^N \left[ -\eta_t y_t^2 - \frac{1}{2} \ln(-\eta_t) \right] + \sum\limits_{t = 2}^N [\lambda \tau_t + \mu_t^1(\eta_t - \eta_{t - 1} - \tau_t) + \mu_t^2(\eta_{t - 1} - \eta_t - \tau_t)].
\end{align*}
Notice that we have not included Lagrange multipliers associated with the constraints $\eta_t < 0$, since the optimal solutions of \eqref{eq:30} cannot satisfy $\eta_t = 0$ (otherwise the cost would be infinite, due to the logarithms $\ln(-\eta_t)$). Therefore, the KKT conditions associated with \eqref{eq:30} are
\begin{align} \label{eq:31}
-y_1^2 - \frac{1}{2 \eta_1} - \mu_2^1 + \mu_2^2 &= 0 \nonumber \\
-y_t^2 - \frac{1}{2 \eta_t} + \mu_t^1 - \mu_{t + 1}^1 - \mu_t^2 + \mu_{t + 1}^2 &= 0, \quad t = 2, \ldots, N - 1 \nonumber \\
-y_N^2 - \frac{1}{2 \eta_N} + \mu_N^1 - \mu_N^2 &= 0 \nonumber \\
\lambda - \mu_t^1 - \mu_t^2 &= 0, \quad t = 2, \ldots, N \nonumber \\
\eta_t &< 0, \quad t = 1, \ldots, N \nonumber \\
\eta_t - \eta_{t - 1} &\leqslant \tau_t, \quad t = 2, \ldots, N \\
\eta_{t - 1} - \eta_t &\leqslant \tau_t, \quad t = 2, \ldots, N \nonumber \\
\mu_t^1 &\geqslant 0, \quad t = 2, \ldots, N \nonumber \\
\mu_t^2 &\geqslant 0, \quad t = 2, \ldots, N \nonumber \\
\mu_t^1(\eta_t - \eta_{t - 1} - \tau_t) &= 0, \quad t = 2, \ldots, N \nonumber \\
\mu_t^2 (\eta_{t - 1} - \eta_t - \tau_t) &= 0, \quad t = 2, \ldots, N. \nonumber
\end{align}
%(31)
From the fourth set of conditions, we have that $\mu_t^2 = \lambda - \mu_t^1$ for $t = 2, \ldots, N$. In addition, from \eqref{eq:30} it can be seen that equality has to be achieved for each $\tau_t$ either for the sixth or seventh set of conditions in \eqref{eq:31}. Hence, \eqref{eq:31} can be simplified to
\begin{align} \label{eq:32}
-y_1^2 - \frac{1}{{2{\eta _1}}} + \lambda  - 2\mu _2^1 &= 0 \nonumber \\
- y_t^2 - \frac{1}{{2{\eta _t}}} + 2\mu _t^1 - 2\mu _{t + 1}^1 &= 0,\quad t = 2, \ldots ,N - 1 \nonumber \\
- y_N^2 - \frac{1}{{2{\eta _N}}} + 2\mu _N^1 - \lambda &= 0 \nonumber \\
{\eta _t} &< 0,\quad t = 1, \ldots ,N \\
\left| {{\eta _t} - {\eta _{t - 1}}} \right| &= {\tau _t},\quad t = 2, \ldots ,N \nonumber \\
\mu _t^1 &\in [0,\lambda ],\quad t = 2, \ldots ,N \nonumber \\
\mu _t^1({\eta _t} - {\eta _{t - 1}} - {\tau _t}) &= 0,\quad t = 2, \ldots ,N  \nonumber \\
(\lambda  - \mu _t^1)({\eta _{t - 1}} - {\eta _t} - {\tau _t}) &= 0,\quad t = 2, \ldots ,N. \nonumber
\end{align}
%(32)
Now, from the last three sets of conditions in \eqref{eq:32} it follows that
\begin{align} \label{eq:33}
\mu_t^1 & \left\{ \begin{array}{ll}
= 0, & \text{if  } \eta_t < \eta_{t - 1} \\
= \lambda, & \text{if  } \eta_t > \eta_{t - 1} \\
\in [0, \lambda], & \text{if  } \eta_t = \eta_{t - 1}.
\end{array} \right.
\end{align}
%(33)
If we let $\tilde{\mu}_t^1 = 2 \mu_t^1 - \lambda$ for all $t = 2, \dots, N$, conditions \eqref{eq:32} can be posed as
\begin{align*}
-y_1^2 - \frac{1}{2 \eta_1} &= \tilde{\mu}_2^1 \\
-y_t^2 - \frac{1}{2 \eta_t} &= \tilde{\mu}_{t + 1}^1 - \tilde{\mu}_t^1, \quad t = 2, \ldots, N - 1 \\
-y_N^2 - \frac{1}{2 \eta_N} &= -\tilde{\mu}_N^1 \\
\eta_t &< 0, \quad t = 1, \ldots, N,
\end{align*}
subject to \eqref{eq:33}. These conditions correspond to \eqref{eq:4}, which concludes the proof. \qed

% -----------------------------------------------------
\subsection{Proof of Lemma~\ref{lem:2}}
This lemma follows by comparing Lemma~\ref{lem:1} with \citep{Rinaldo-09}[equation (2.1) + Lemma A.1]. However, according to Lemma~\ref{lem:1}, the solution of problem~\eqref{eq:2} is required to satisfy $\ve{\sigma} > 0$, hence we need to show that the solution of \eqref{eq:6} necessarily satisfies $\ve{\sigma} > 0$. To this end, let us assume the opposite, \ie let $S := \{t \in \{2, \ldots, N\}:\;\sigma_t^2 \ne \sigma_{t - 1}^2\} \ne \emptyset$. Consider the vector   $\sigma^+ \in \mathbb{R}^N$ given by $(\sigma^+)_t = \max\{\sigma_t, 0\}$ for $t = 1, \dots, N$. It is easy to see that the cost in \eqref{eq:6} for $\ve{\sigma}^+$ is strictly lower than for $\ve{\sigma}$, which contradicts the optimality of the latter. This concludes the proof. \qed

% ==============================

\bibliographystyle{acmtrans-ims}
\bibliography{cristian,refs_bo}

\end{document}